\documentclass[11pt]{article}

\usepackage[margin=1in]{geometry}

\usepackage[round]{natbib}

\usepackage{amsmath,amssymb,amsthm,mathtools}
\usepackage{booktabs}
\usepackage{graphicx}
\usepackage{xcolor}
\usepackage{tikz}
\usepackage{url}
\usepackage{hyperref}

\usetikzlibrary{arrows.meta,calc,positioning,fit}
\usepackage{microtype}
\usepackage{enumitem}
\newcommand{\median}{\operatorname{median}}

\newtheorem{theorem}{Theorem}
\newtheorem{proposition}[theorem]{Proposition}
\newtheorem{lemma}[theorem]{Lemma}
\newtheorem{corollary}[theorem]{Corollary}
\newtheorem{assumption}[theorem]{Assumption}
\theoremstyle{definition}

\theoremstyle{remark}
\newtheorem{remark}[theorem]{Remark}

\newcommand{\E}{\mathbb{E}}
\newcommand{\Prob}{\mathbb{P}}
\newcommand{\R}{\mathbb{R}}
\newcommand{\cG}{\mathcal{G}}
\newcommand{\cR}{\mathcal{R}}
\newcommand{\Kaff}{\mathfrak{K}^{\mathrm{aff}}_T(f;n)}
\newcommand{\ind}{\iota}
\newcommand{\prox}{\operatorname{prox}}

\definecolor{carpblue}{RGB}{45,105,160}
\definecolor{carpteal}{RGB}{42,135,125}
\definecolor{carporange}{RGB}{213,120,45}
\definecolor{carppurple}{RGB}{122,88,160}
\definecolor{carplight}{RGB}{239,244,248}
\newcounter{algorithm}
\renewcommand{\thealgorithm}{\arabic{algorithm}}
\newenvironment{algorithmblock}[1]{%
  \refstepcounter{algorithm}\begin{figure}[t]\small
  \noindent\textbf{Algorithm \thealgorithm: #1}\par\smallskip\hrule\smallskip
}{\smallskip\hrule\end{figure}}

\begin{document}

\title{Adaptive Restart under Unknown KL Geometry in Stochastic Weakly-Convex Optimization}
\author{
Binh T. Nguyen$^{1}$ \qquad
Khoa D. Nguyen$^{1}$ \qquad
Bao Thai Huynh$^{2}$\\[1em]
\small $^{1}$University of Science,
Vietnam National University Ho Chi Minh City, Vietnam\\[0.4em]
\small $^{2}$University of Information Technology,
Vietnam National University Ho Chi Minh City, Vietnam
}
\date{}
\maketitle
\begin{abstract}
KL geometry accelerates restart schedules in stochastic weakly-convex optimization, but the schedule depends jointly on an unknown and potentially nonunique geometric description. We introduce CARP-KL, a certified portfolio that adapts over valid KL descriptions while using the same primitive stochastic-subgradient oracle for both optimization and candidate selection. Over declared polynomial parameter ranges, CARP-KL competes with every affordable valid description and preserves the known-description polynomial Moreau-stationarity rate up to logarithmic factors, without objective values, an exact proximal map, direct Moreau-gradient access, or a bounded-domain assumption. Controlled experiments quantify the finite-budget price of this universality. The portfolio repeatedly contains strong candidates, while certification consumes a substantial share of the oracle budget and finite-sample ranking limits their recovery. A common conservative gap bound further reduces the portfolio from $O(\log^3 T)$ to $O(\log^2 T)$ candidates with an explicit scale-mismatch penalty. The theory and diagnostics together establish certified unknown-geometry adaptation and pinpoint selection as the principal practical frontier of the current construction.
\end{abstract}

% BEGIN sections/01_introduction
\section{Introduction}
\label{sec:intro}

Stochastic subgradient methods for general weakly-convex objectives attain
$O(T^{-1/4})$ RMS Moreau stationarity after $T$ oracle calls
\citep{davis2019stochastic}. Additional KL geometry can yield faster restart
rates. In the nonsmooth stochastic weakly-convex setting, the recent ERdA
analysis of \citet{elhishi2026restart} attains
\[
  T^{-\nu_\theta},
  \qquad
  \nu_\theta=\frac{\theta}{4\theta-1}\in[1/3,1/2],
\]
and shows substantial robustness to KL-dependent schedule misspecification.

We study a complementary adaptation problem. A KL restart schedule depends on
a description $\kappa=(\theta,c,\Delta)$ comprising the exponent, KL constant,
and an initial-gap bound. Such descriptions may be unknown and nonunique, with
different valid descriptions inducing different finite-horizon guarantees.
Moreover, after running several schedules, their outputs cannot be ranked
directly from the primitive stochastic-subgradient oracle: Moreau stationarity
would require proximal information, while objective gaps require $f^\star$.
Thus adaptation requires both competition over unknown descriptions and
observable candidate selection under the original oracle budget.

We introduce \emph{CARP-KL}, a certified adaptive restart portfolio. It uses a
one-sided grid of KL descriptions, generates their restart candidates, and
ranks them using fresh Moreau-stationarity certificates constructed from the
same primitive oracle. Candidate generation, certification, and selection are
all charged to one horizon of $T$ stochastic-subgradient calls.

\paragraph{Contributions.}
\begin{enumerate}[leftmargin=*,itemsep=1pt,topsep=2pt]

\item \textbf{Competition over nonunique KL descriptions.}
CARP-KL competes with every valid affordable
$\kappa=(\theta,c,\Delta)$ rather than estimating a distinguished ``true''
exponent. A validity-preserving one-sided grid has
$O((\log T)^3)$ candidates over polynomial parameter ranges. If a common
conservative initial-gap bound is available, the gap coordinate can instead
be shared, reducing the portfolio to $O((\log T)^2)$ candidates with an
explicit scale-mismatch penalty.

\item \textbf{Same-oracle stationarity certification.}
Using fresh projected stochastic-subgradient calls, we obtain
$O(L/\sqrt q)$ Moreau-stationarity certificates without bounded-domain,
objective-value, $f^\star$, exact-proximal, or Moreau-gradient access. Median
amplification enables simultaneous comparison of the random portfolio.

\item \textbf{One-budget adaptive oracle inequality.}
Charging optimization and certification to the same $T$ primitive calls,
CARP-KL preserves, for every fixed valid affordable description, the
known-description polynomial rate up to logarithmic factors:
\[
  \E\|\nabla f_{1/(2\rho)}(\widehat x_T)\|
  =
  \widetilde O\!\left(T^{-\theta/(4\theta-1)}\right).
\]
Thus the contribution is not a new tuned KL exponent, but an executable
same-oracle mechanism for adapting over unknown KL descriptions.
\end{enumerate}

Controlled experiments diagnose the finite-budget cost of this adaptation.
They show substantial budget-dependent behavior: reducing portfolio size
reallocates calls to optimization and certification, while certified selection
remains a principal observed finite-budget bottleneck. We therefore use the
experiments to probe the mechanisms in the oracle inequality, not to claim
empirical verification of its asymptotic exponent.
% END sections/01_introduction

% BEGIN sections/02_related_work
\section{Related Work and Positioning}
\label{sec:related}

\paragraph{Weakly-convex optimization and KL restart.}
For general stochastic weakly-convex objectives, stochastic subgradient
methods attain $O(T^{-1/4})$ RMS Moreau stationarity
\citep{davis2019stochastic}; adaptive methods such as AdaGrad address other
forms of unknown scale but do not directly exploit KL geometry
\citep{ward2019adagrad}. Restart and sharpness have long accelerated convex
nonsmooth optimization \citep{yang2018rsg}, and KL geometry yields faster
rates under additional structure \citep{fatkhullin2022sharp}. Closest to our
setting, ERdA \citep{elhishi2026restart} obtains the optimal
known-geometry polynomial rate for nonsmooth stochastic weakly-convex
optimization and shows substantial robustness to KL-dependent schedule
misspecification, including near-adaptivity with only a $\log\log T$ loss.

Our question is complementary. We consider unknown and potentially nonunique
KL descriptions and ask how to select among their candidate outputs when
objective gaps, $f^\star$, exact proximal maps, and direct Moreau gradients
are unavailable. In particular, the ERdA constant-restart prescription uses
the current Moreau-envelope gap to $f^\star$, whereas CARP-KL uses only the
primitive stochastic-subgradient oracle for both optimization and fresh
candidate certification, charging all calls to one horizon.

\paragraph{Unknown geometry and parameter-free adaptation.}
Unknown sharpness and growth parameters can be handled by search, restart, or
parallel schemes in convex optimization
\citep{roulet2017sharpness,renegar2022restart,fercoq2019adaptive,xu2017adaptive},
including joint uncertainty in several sharpness parameters
\citep{adcock2025restarts}. More general parameter-free methods use
hyperparameter search, certification, or scale adaptation
\citep{carmon2022making,attia2024free,orabona2016coin,jun2019noise}, and
adaptation itself can incur nonconstant cost
\citep{carmon2024price}. These precedents make neither restart nor unknown
parameter adaptation new by itself. The additional issue here is the
validation interface: KL descriptions need not be unique, and Moreau
stationarity is not directly observable from the primitive oracle.

\paragraph{Positioning of CARP-KL.}
CARP-KL competes directly with the best valid affordable KL description,
without requiring a distinguished ``true'' exponent to be identified.
Candidate generation, fresh certification, and final selection all use the
same primitive stochastic-subgradient oracle and are charged to a single
total budget. The resulting oracle inequality preserves the
known-description polynomial stationarity rate up to logarithmic factors.
When a common conservative initial-gap bound is available, CARP-KL shares
this bound across schedules, reducing the portfolio from
$O((\log T)^3)$ to $O((\log T)^2)$ candidates with an explicit
scale-mismatch penalty. Together, these results establish an executable,
same-oracle framework for observable competition among unknown and
potentially nonunique KL descriptions.

\paragraph{Empirical comparators.}
Our experiments include geometry-informed Oracle-KL, tuned SGD, LogSearch
motivated by parameter-free hyperparameter search \citep{attia2024free},
AdaGrad-Norm, and a Gao--Deng specialization for stochastic weakly-convex
optimization \citep{gao2024weaklyconvex}. Related methods such as U-DoG and
Free-FISTA study different convex adaptation regimes
\citep{kreisler2024udog,aujol2024freefista} and are contextual rather than
like-for-like comparators.
% END sections/02_related_work

% BEGIN sections/03_setting
\section{Setting and Why KL Adaptation Is Nontrivial}
\label{sec:setting}

Let $X\subseteq\R^d$ be nonempty, closed, and convex, possibly unbounded, and
consider
\[
  \min_{x\in X} f(x),\qquad f=g+\ind_X,
\]
where $g$ is proper, lower semicontinuous, and $\rho$-weakly convex. The
primitive stochastic oracle satisfies, uniformly on $X$,
\begin{equation}
  \E_\xi G(x,\xi)\in\partial g(x),
  \qquad
  \E_\xi\|G(x,\xi)\|^2\le L^2.
  \tag{O}\label{eq:oracle}
\end{equation}
We assume $\rho,L$, and the horizon $T$ are known. Define
\[
  F=f_{1/(2\rho)},\qquad
  S(x)=\|\nabla F(x)\|,\qquad
  \Delta_0=F(x_0)-f^\star .
\]
Writing $\lambda=1/(2\rho)$,
\[
  F(x)=\min_{y\in X}\{g(y)+\rho\|y-x\|^2\},
  \nabla F(x)=2\rho(x-\prox_{\lambda f}(x)).
\]
Weak convexity makes the proximal subproblem $\rho$-strongly convex. These
quantities define the performance criterion, not an additional oracle:
CARP-KL uses only \eqref{eq:oracle} for candidate construction and selection;
projections onto $X$ are deterministic operations.

\begin{assumption}[Common-region pointwise KL geometry]
\label{ass:klc}
There exists $\mathcal V\subseteq X$ containing $x_0$ and every iterate of
each neighboring-grid restart trajectory used in the oracle comparison. For
at least one $\kappa=(\theta,c,\Delta)$ with
\[
  \theta\in[1/2,1],\qquad c>0,\qquad
  0<\Delta_0\le\Delta\le1,
\]
we have
\[
  S(x)\ge c(F(x)-f^\star)^\theta
  \qquad\text{for all }x\in\mathcal V .
\]
\end{assumption}

The normalization $\Delta\le1$ fixes the safe discretization directions:
along the expected-gap restart induction, increasing $\theta$, decreasing
$c$, and increasing the gap upper bound preserve the recursion. This is an
expectation-level property and does not assert a realized gap bound.
The common-region condition is needed because neighboring descriptions induce
different random trajectories; it is automatic when the KL inequality holds
globally on $X$. CARP-KL does not test region membership.

\paragraph{Comparator class.}
A triple $\kappa=(\theta,c,\Delta)$ is a \emph{valid KL description} if it
satisfies the parameter restrictions above and the KL inequality on
$\mathcal V$, including the common-region requirement for its neighboring
grid trajectory. An objective may admit many valid descriptions, so our
comparator is a class of descriptions rather than a distinguished parameter.
The algorithm receives declared ranges
\[
  \theta\in[1/2,1],\qquad
  c\in[c_{\min,T},c_{\max,T}],\qquad
  \Delta\in[\Delta_{\min,T},1],
\]
with
$
  c_{\max,T}/c_{\min,T}\le T^{r_c},
$ $
  \Delta_{\min,T}=T^{-r_\Delta},
$ and $
  r_c,r_\Delta\ge0,
$
where $r_c$ and $r_\Delta$ are fixed range-growth exponents. These
polynomial ranges ensure logarithmic grid sizes. The declared box may contain
invalid descriptions; validity is an analytical comparator property, not
algorithmic input.

A valid description is additionally \emph{affordable} when its relevant
first-stage cost fits the per-candidate training cap, as formalized in
Section~\ref{sec:results}. For every fixed admissible description this
restriction eventually disappears as $T\to\infty$.

\paragraph{Training and certification.}
Candidate outputs are measurable with respect to a training sigma-field
$\mathcal T$; certification uses fresh samples conditional on $\mathcal T$.
This yields simultaneous conditional concentration without requiring
independence among candidate training runs.

\subsection{Why adaptation is nontrivial}
\label{sec:setting:failure}

Four issues make adaptation nontrivial. First, while $\theta$ determines the
asymptotic exponent, $c$ and $\Delta$ control finite-horizon restart scales:
optimistic values can invalidate the recursion, whereas conservative ones can
be unaffordable. Adaptation must therefore cover the joint description
$(\theta,c,\Delta)$.

Second, discretization must preserve the restart recursion. For gaps in
$(0,1)$, the safe directions are $\theta$ upward, $c$ downward, and $\Delta$
upward, so ordinary nearest-neighbor rounding is insufficient.

Third, candidate quality is not directly observable from \eqref{eq:oracle}:
objective gaps require $f^\star$, while $S(x)$ requires unavailable
Moreau-gradient or exact-proximal information. CARP-KL instead uses fresh
same-oracle stationarity certificates.

Finally, all adaptation must fit within the original horizon. CARP-KL charges
candidate construction, certification, and selection to at most $T$ primitive
calls, rather than granting each candidate a separate budget or a stronger
validation oracle.

The resulting oracle comparator therefore ranges over \emph{valid and
affordable} KL descriptions: those preserving the required recursion and
supporting an executable finite-horizon restart guarantee.
% END sections/03_setting

% BEGIN sections/04_algorithm
\section{CARP-KL}
\label{sec:algorithm}

CARP-KL turns the oracle-competition problem of
Section~\ref{sec:setting} into two coupled operations: construct a portfolio
that contains a useful one-sided approximation to every affordable valid KL
description, and select among the resulting random candidates using only the
primitive oracle. Both operations share the same horizon.

\paragraph{One-sided portfolio.}
Let
\[
  \eta_T=\frac{1}{\lceil\log(eT)\rceil}.
\]
We round $\theta$ upward on an $\eta_T$-mesh, $c$ downward on a dyadic grid,
and $\Delta$ upward on a dyadic grid. Let $\cG_T$ denote the product of these
declared grids and let $K_T=|\cG_T|$. Under the polynomial endpoint ranges of
Section~\ref{sec:setting},
\[
  K_T=O((\log T)^3).
\]
The directions of rounding are chosen to preserve the restart recursion, not
merely to approximate the numerical values of $(\theta,c,\Delta)$.

For confidence level $\delta\in(0,1/2)$, CARP-KL allocates half of the
primitive-call horizon to candidate construction and half to certification,
and shares each half uniformly across the portfolio:
\begin{align*}
  n=m&=\left\lfloor\frac{T}{2K_T}\right\rfloor,
  B_K=2\left\lceil4\log(K_T/\delta)\right\rceil+1,\\
  q_T&=\left\lfloor\frac{m}{B_K}\right\rfloor.
\end{align*}

\begin{algorithmblock}{CARP-KL: certified adaptive restart portfolio}
\label{alg:carp}
\textbf{Input:} $x_0,T,\delta,\rho,L$ and declared ranges for
$(\theta,c,\Delta)$.\\
1. Construct the one-sided joint grid $\cG_T$.\\
2. For every $\kappa_j\in\cG_T$, execute each complete tuned restart stage
whose prescribed calls fit under cap $n$, and return the last completed output
$x_j$ (or $x_0$ if no stage fits).\\
3. For each $x_j$, use fresh samples to run $B_K$ independent certification
blocks of $q_T$ calls. Let
\[
  Z_{j,b}=2\rho\|x_j-y_{q_T}^{(j,b)}\|,
  \qquad
  \widehat S_j=\median_{1\le b\le B_K} Z_{j,b}.
\]
4. Choose
\[
  \widehat j\in
  \arg\min_{1\le j\le K_T}
  \bigl\{\widehat S_j+r_{q_T}\bigr\},
\]
and return $\widehat x_T=x_{\widehat j}$.
\end{algorithmblock}

\paragraph{Candidate construction.}
For a grid description $\kappa_j$, Step~2 executes restart stages atomically:
a stage begins only if its full prescribed integer length fits within the
training cap $n$. Hence, every grid point has a well-defined returned
candidate, including $x_0$ when no stage fits. CARP-KL never tests whether
$\kappa_j$ is valid. Invalid or unaffordable descriptions may generate poor
outputs; the analysis requires only that the portfolio contain the
recursion-valid affordable neighbor of each comparator description.

\paragraph{Same-oracle certification.}
For certification block $b$ centered at $x_j$, initialize
$y_0^{(j,b)}=x_j$ and iterate
\[
\begin{aligned}
  y_{t+1}^{(j,b)}
  =
  \Pi_X\!\Biggl[
    y_t^{(j,b)}
    -\frac{2}{\rho(t+32)}
    \Bigl(
      &G(y_t^{(j,b)},\xi_t) \\
      &+2\rho\bigl(y_t^{(j,b)}-x_j\bigr)
    \Bigr)
  \Biggr].
\end{aligned}
\]
for $t=0,\ldots,q_T-1$. The block output
\[
  Z_{j,b}
  =
  2\rho\|x_j-y_{q_T}^{(j,b)}\|
\]
approximates
\[
  S(x_j)
  =
  2\rho
  \bigl\|x_j-\prox_{f/(2\rho)}(x_j)\bigr\|.
\]
The proximal point on the right is never evaluated by the algorithm; it
defines the target estimated through the stochastic proximal subproblem.
With
\[
  r_q=\frac{8\sqrt{66}\,L}{\sqrt{q+32}},
\]
the median construction yields simultaneous candidate-wise confidence
intervals from fresh primitive-oracle calls.

On this simultaneous certificate event, the upper-confidence selector obeys
the deterministic comparison
$
  S(x_{\widehat j})
  \le
  S(x_j)+2r_{q_T}
$, $
  \text{for every }j.
$
Thus, selection reduces the adaptive problem to two quantities: the quality of
the best candidate generated by the portfolio and the resolution
$r_{q_T}$ with which the same oracle can distinguish among candidates.

\paragraph{Exact primitive-call accounting.}
The portfolio and its validation fit inside the original horizon. Since
$B_Kq_T\le m$,
\begin{equation}
  K_Tn+K_TB_Kq_T
  \le
  K_T(n+m)
  \le T.
  \label{eq:budget}
\end{equation}
Unused restart-cap calls and certification remainders are deliberately not
reassigned. Consequently, every term in the final oracle inequality
corresponds to calls available to the stated algorithm rather than to an
external validation resource.

The certification blocks are independent across $b$ conditional on the
training sigma-field $\mathcal T$; the candidate training runs themselves need
not be independent. This is the only sample-splitting condition required by
the selection argument. The construction is horizon-aware and assumes
$q_T\ge1$, equivalently the finite-budget condition $m\ge B_K$ appearing in
Theorem~\ref{thm:master}.

\subsection{Proof-critical design choices}
\label{sec:algorithm:invariants}

The following choices are part of the theoretical construction rather than
implementation conventions.

\paragraph{Atomic restart stages.}
The external cap is not a stopping time within a restart stage. Truncating a
stage when the cap is reached would change the output rule controlled by the
base restart lemma. CARP-KL therefore checks the prescribed stage length
before execution and runs only complete stages; residual calls may remain
unused.

\paragraph{Validity-agnostic execution.}
The product grid intentionally contains descriptions that may not satisfy the
KL inequality for the current objective. CARP-KL executes them without a
validity test. The oracle analysis identifies a recursion-valid neighboring
description only for comparison, and does not require the remaining portfolio
members to be valid. This is essential to competition over a nonunique
description class: the algorithm searches outputs, not KL parameters.

\paragraph{Fresh conditional certification.}
After candidate construction, condition on the entire trained portfolio.
Each certification block then has a fixed center $x_j$ and fresh oracle
variables. Blocks may be executed in parallel, but the present proof does not
reuse training samples for certification. The odd block count $B_K$ ensures
that whenever a strict majority of blocks are accurate, their scalar median is
accurate as well.

\paragraph{Scalar rather than vector certification.}
Selection requires only the stationarity magnitude $S(x_j)$, not the direction
of the Moreau gradient. Estimating the scalar distance
$
  2\rho\|x_j-y^\star_j\|
$
therefore avoids dimension-dependent vector concentration. This reduction is
also compatible with the second-moment oracle assumption
\eqref{eq:oracle}: the certificate uses constant-probability block control
followed by median amplification, rather than assuming sub-Gaussian oracle
noise.

\paragraph{Upper-confidence selection.}
The stated algorithm assigns the same certification budget to every candidate,
so all radii equal $r_{q_T}$. Writing the selector in upper-confidence form
makes the proof interface more general: if candidate $j$ had a valid radius
$r_j$, the same deterministic argument would give
$
  S(x_{\widehat j})
  \le
  S(x_j)+2r_j,
  \text{ for every }j
$
on the joint certificate event. Unequal or adaptive certification allocations
could therefore replace uniform allocation provided they preserve valid
candidate-specific radii and the total primitive-call budget.

\begin{table}[t]
\caption{Proof interfaces in the end-to-end CARP-KL guarantee.}
\label{tab:interfaces}
\centering
\small
\begin{tabular}{@{}p{0.23\columnwidth}p{0.31\columnwidth}p{0.34\columnwidth}@{}}
\toprule
Module & Input & Exported guarantee \\
\midrule
Restart
  & valid $\kappa$, cap $n$
  & RMS bound $\cR_n(\kappa)$ \\
Grid
  & continuous $\kappa$
  & valid affordable neighbor \\
Certificate
  & $x_j$, fresh calls
  & radius $r_{q_T}$ \\
Selection
  & simultaneous event
  & best candidate $+\,2r_{q_T}$ \\
Accounting
  & floors and caps
  & at most $T$ calls \\
\bottomrule
\end{tabular}
\end{table}
% END sections/04_algorithm

% BEGIN sections/05_main_results
\section{Main Results}
\label{sec:results}

For $\kappa=(\theta,c,\Delta)$, define
\[
p_\theta=4\theta-1,\quad
b_\theta=e^{1-1/(4\theta)}-1,\quad
\nu_\theta=\frac{\theta}{4\theta-1},
\]
and
\[
A_\kappa=16e\,c^{-4}\rho L^2\Delta^{-p_\theta},
\cR_u(\kappa)
=e^{1/4}c\Delta^\theta
\left[1+\frac{b_\theta u}{1+A_\kappa}\right]^{-\nu_\theta}.
\]
Here $A_\kappa$ is the first-stage scale and
$\cR_u(\kappa)=O(u^{-\nu_\theta})$ is the known-description bound under cap
$u$. Retaining this nonasymptotic form is essential because affordability and
portfolio sharing depend on the full description, not only on its exponent.

\begin{proposition}[Capped restart]
\label{prop:restart}
If the first tuned stage for $\kappa$ fits under cap $u$, let
$N(u,\kappa)$ denote the number of completed restart stages. The last completed
restart output satisfies
\[
  \bigl(\E S(x_{N(u,\kappa)})^2\bigr)^{1/2}
  \le \cR_u(\kappa).
\]
\end{proposition}

The capped statement accounts for integer stage lengths and is proved in
Appendix~\ref{app:restart}.

\begin{lemma}[KL-grid approximation]
\label{lem:grid}
Under the declared polynomial parameter ranges,
$K_T=O((\log T)^3)$, and every valid description $\kappa$ has a
recursion-valid neighbor $\bar\kappa\in\cG_T$ satisfying
\[
  A_{\bar\kappa}\le C_AA_\kappa,
  \qquad
  \cR_u(\bar\kappa)\le C_{\rm grid}\cR_u(\kappa),
\]
where
$
  C_A=16e^{4r_\Delta}
$ and $
  C_{\rm grid}\le8e^{2+2r_\Delta}.
$
\end{lemma}

The one-sided grid preserves recursion validity, while the
$O(1/\log T)$ exponent mesh prevents a polynomial discretization loss;
see Appendix~\ref{app:grid}.

\begin{proposition}[Primitive-oracle Moreau certificate]
\label{prop:certificate}
For any candidate $x$, using
\[
  B(\varepsilon)=2\lceil4\log(1/\varepsilon)\rceil+1
\]
independent blocks of $q$ fresh primitive calls gives
\[
  \Prob\!\left(
  |\widehat S(x)-S(x)|
  >
  \frac{8\sqrt{66}L}{\sqrt{q+32}}
  \,\middle|\,x
  \right)
  \le\varepsilon.
\]
No boundedness of $X$, exact proximal map, or direct Moreau-gradient query is
required.
\end{proposition}

The conditional statement applies to random trained candidates; the
self-localizing proximal recursion and median amplification are proved in
Appendix~\ref{app:certificate}.

\paragraph{Affordable comparators.}
Let
\[
  \Kaff=
  \left\{
  \begin{aligned}
    \kappa:\;&
    \kappa\text{ is valid and in the declared ranges},\\
    &C_A(1+A_\kappa)\le n
  \end{aligned}
  \right\}.
\]
The factor $C_A$ ensures that the grid neighbor of each comparator can complete
its first restart stage within cap $n$.

\begin{theorem}[Same-oracle adaptation to unknown KL geometry]
\label{thm:master}
Under Section~\ref{sec:setting} and Assumption~\ref{ass:klc}, suppose the
declared ranges are polynomial, $\delta\in(0,1/2)$,
\[
  T\ge4K_T,\qquad m\ge B_K,\qquad \Kaff\ne\varnothing.
\]
Then Algorithm~\ref{alg:carp} uses at most $T$ primitive stochastic-subgradient
calls and
\[
\begin{aligned}
  \E S(\widehat x_T)
  &\le
  C_{\rm grid}
  \inf_{\kappa\in\Kaff}
  (4K_T)^{\nu_\theta}\cR_T(\kappa)
  \\
  &\quad
  +\frac{16\sqrt{66}L}{\sqrt{q_T+32}}
  +4L\delta .
\end{aligned}
\]
\end{theorem}

The three terms correspond respectively to optimization with portfolio
sharing, certified selection, and certificate failure. Importantly, the
infimum ranges over complete KL descriptions rather than a canonical exponent:
different valid descriptions may trade a better asymptotic exponent against a
smaller finite-horizon restart scale. The theorem therefore performs the
comparison at the level of executable finite-horizon guarantees. Details of
conditional selection and the cap-to-horizon comparison appear in
Appendix~\ref{app:master}.

\begin{corollary}[Rate adaptation]
\label{cor:rate}
For every fixed admissible $\kappa=(\theta,c,\Delta)$ that eventually lies in
the declared ranges, with $\delta=T^{-1}$,
\[
  \E\|\nabla f_{1/(2\rho)}(\widehat x_T)\|
  =
  \widetilde O\!\left(
  T^{-\theta/(4\theta-1)}
  \right).
\]
\end{corollary}

Indeed,
$
  q_T=\Omega(T/\log^4T),
$
so certification is $O(L\log^2T/\sqrt T)$, while portfolio-shared optimization
is
$
  O\!\left(
  \log^{3\nu_\theta}T\,T^{-\nu_\theta}
  \right).
$
Since $\nu_\theta\in[1/3,1/2]$, certification does not worsen the polynomial
exponent. Thus adaptation preserves the known-description polynomial rate;
its additional cost appears only through logarithmic portfolio-sharing and
certification factors.

\paragraph{Sharing a common initial-gap bound.}
The gap coordinate need not be searched independently when a common
conservative bound is available.

\begin{corollary}[Reduced portfolio under a common gap bound]
\label{cor:reduced}
Suppose $\Delta_+\in(0,1]$ is a common conservative bound and let
$\cG_T^{\rm red}$ be the two-dimensional one-sided grid over $(\theta,c)$
obtained by sharing $\Delta_+$ across all candidates. Then
\[
  K_T^{\rm red}=|\cG_T^{\rm red}|=O((\log T)^2).
\]
For a valid reduced-affordable
$\kappa=(\theta,c,\Delta)$ with $\Delta\le\Delta_+$, let
$\Gamma_\Delta=\Delta_+/\Delta$ and let $\bar\theta$ be the upper grid
neighbor of $\theta$. The optimization term in the same-oracle inequality
becomes
\[
  4e^{2+2r_\Delta}
  \inf_{\kappa}
  \Gamma_\Delta^{\bar\theta}
  (4K_T^{\rm red})^{\nu_\theta}
  \cR_T(\kappa),
\]
where the infimum is over the reduced affordable comparator class.
If $\Gamma_\Delta=O(1)$, the portfolio-sharing factor improves from
\[
  O(\log^{3\nu_\theta}T)
  \quad\text{to}\quad
  O(\log^{2\nu_\theta}T),
\]
and, for $\delta=T^{-1}$, certification is
$O(L\log^{3/2}T/\sqrt T)$.
\end{corollary}
Thus, a shared gap bound trades independent gap adaptation for a smaller
portfolio, while paying the explicit mismatch
$\Gamma_\Delta^{\bar\theta}$. The reduction is therefore structural rather
than uniformly beneficial: a looser shared bound decreases portfolio size but
can simultaneously worsen the comparator through the mismatch factor. The
exact comparison is in Appendix~\ref{app:grid}.

\paragraph{Scope.}
The guarantee is horizon-aware and controls expected stationarity with known
$\rho$ and $L$. High-probability tuned optimization, unknown oracle-scale
adaptation, and minimax optimality of the adaptive overhead remain outside the
present result.
% END sections/05_main_results

% BEGIN sections/06_proof_ideas
\section{Proof Ideas and Rate Comparison}
\label{sec:proofideas}

The proof composes four interfaces; complete arguments appear in
Appendices~\ref{app:restart}--\ref{app:master}. First, the known-description
restart bound is written under an external call cap $u$, so atomic execution
yields $\cR_u(\kappa)$ for every affordable valid description. This cap-level
form is essential because each portfolio candidate receives only a fraction
of the full horizon. Second, one-sided rounding
$
  \bar\theta\ge\theta,
  \bar c\le c,
  \bar\Delta\ge\Delta
$
preserves the restart recursion, while the $O(1/\log T)$ exponent mesh gives
\[
  A_{\bar\kappa}\le C_AA_\kappa,
  \qquad
  \cR_u(\bar\kappa)\le C_{\rm grid}\cR_u(\kappa).
\]
Thus, every affordable continuous comparator has a recursion-valid grid
neighbor whose finite-horizon guarantee changes by only a constant factor.

Third, conditioned on a trained candidate, fresh stochastic projected descent
on its strongly convex proximal subproblem, with proximal solution
$y_x^\star=\operatorname{prox}_{f/(2\rho)}(x)$, gives
\[
  \E\|y_q-y_x^\star\|^2
  =O\!\left(\frac{L^2}{\rho^2q}\right),
\]
without boundedness of $X$. Because certification uses fresh samples
conditional on the trained portfolio, the candidates themselves may be random;
median amplification and a union bound then yield simultaneous certificates.
Finally, on this event,
$
  S(x_{\widehat j})\le S(x_j)+2r_{q_T}
$
for every candidate $j$. Combining this deterministic selection inequality
with the grid comparison and $n\ge T/(4K_T)$ yields
Theorem~\ref{thm:master} and
$\widetilde O(T^{-\theta/(4\theta-1)})$ stationarity for every fixed valid
affordable description.

The reduced portfolio applies the same argument with a shared conservative
gap bound $\Delta_+$. Writing
$\Gamma_\Delta=\Delta_+/\Delta$, removing the independent gap grid reduces
the portfolio from $O((\log T)^3)$ to $O((\log T)^2)$ candidates, with the
scale mismatch appearing explicitly as
$\Gamma_\Delta^{\bar\theta}$ in Corollary~\ref{cor:reduced}. The reduction
therefore converts controlled prior scale information into larger
per-candidate optimization and certification budgets, while retaining an
explicit accounting of the gap mismatch.

For rate positioning, generic stochastic subgradient methods provide
$O(T^{-1/4})$ Moreau stationarity, while known-geometry ERdA achieves the
KL-dependent rate $O(T^{-\nu_\theta})$. CARP-KL extends this rate guarantee
to observable competition over unknown and potentially nonunique KL
descriptions, preserving the known-geometry polynomial exponent up to
logarithmic factors while charging candidate construction, certification,
and selection to a single primitive-oracle budget.
% END sections/06_proof_ideas

% BEGIN sections/07_experiments
\section{Empirical Diagnostics}
\label{sec:experiments}

Our experiments study the finite-budget behavior behind the oracle inequality:
how competitive the returned CARP-KL point is as the primitive-call horizon
grows, how much is lost relative to a geometry-informed oracle, and how
portfolio size affects the resources available for optimization and
certification. The goal is diagnostic rather than asymptotic rate
verification. All protocols are frozen before evaluation, every optimization
and certification call is charged to the primitive stochastic-subgradient
budget, and prespecified negative outcomes are retained.

We consider
\[
  f_p(x)=\frac{|x|^p}{p},
  \qquad
  p\in\{2,3,4\},
\]
together with a near-flat endpoint (FLAT), giving a controlled progression of
KL geometries. Appendix~\ref{app:experiments} specifies the objectives,
stochastic oracle, frozen hyperparameters, evaluation seeds, offline
stationarity metric, and additional diagnostics.

\paragraph{Competitive behavior across budgets.}
We compare CARP-KL with Oracle-KL, a geometry-informed known-description
benchmark, together with tuned SGD, Attia--Koren-style LogSearch,
AdaGrad-Norm, and a source-policy-faithful Gao--Deng specialization. The
prospectively frozen budget grid is
$
  \log_2T\in\{15,17,19,20,21,23,25\},
$
with the same 20 evaluation seeds at every budget. Comparator hyperparameters
are selected using disjoint tuning seeds and then frozen before evaluation.
All methods are compared under primitive stochastic-subgradient call
accounting; CARP-KL additionally pays for its certification calls within the
same budget.

Figure~\ref{fig:competitive-scaling} shows the resulting six-method
trajectories. Here, P2--P4 denote the power objectives
\(f_p(x)=|x|^p/p\) with \(p\in\{2,3,4\}\), respectively.
The trajectories reveal a clear finite-budget regime structure. CARP-KL
exhibits a pronounced small-budget transition, followed on P2--P4 by
generally decreasing absolute stationarity across the moderate-to-large
tested budgets, while the tuned and adaptive comparators remain competitive
throughout the sweep. These results characterize the returned performance
of the complete adaptive procedure, with portfolio sharing, fresh
certification, and final selection all charged to the same primitive-oracle
budget. The comparison therefore evaluates the operational cost of
unknown-geometry adaptation rather than an isolated restart schedule.

\begin{figure}[t]
\centering
\includegraphics[width=.81\columnwidth]
{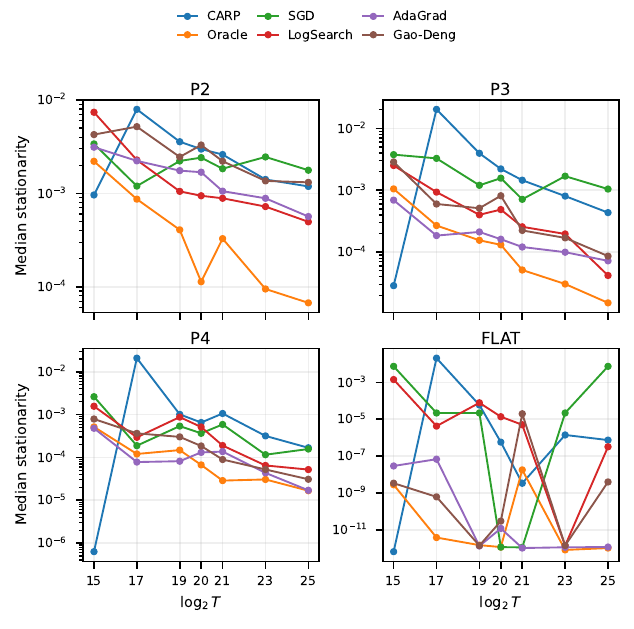}
\caption{\textbf{Competitive behavior across primitive-oracle budgets.}
Median offline Moreau stationarity over 20 frozen evaluation seeds for six
methods at
$\log_2T\in\{15,17,19,20,21,23,25\}$
(lower is better; logarithmic vertical scale).
CARP-KL exhibits a pronounced small-budget transition and is not globally
monotone; its absolute stationarity generally decreases on P2--P4 over the
moderate-to-large tested budgets. FLAT approaches numerical floors and is
interpreted together with dispersion diagnostics in
Appendix~\ref{app:experiments}.}
\label{fig:competitive-scaling}
\end{figure}

\paragraph{Gap to geometry-informed restart.}
To isolate the empirical cost of adaptation from differences among practical
baselines, we use the prespecified seed-paired diagnostic
\[
  G_p(T)
  =
  \operatorname{median}_s
  \log
  \frac{S_{\rm CARP}(T,s)}
       {S_{\rm Oracle}(T,s)}.
\]
The CARP-KL/Oracle-KL gap remains substantial and does not uniformly contract
across the seven tested budgets. The finite-budget results therefore
characterize a persistent separation between the returned adaptive solution
and the geometry-informed oracle over the tested range. This behavior is
consistent with Corollary~\ref{cor:rate}, which preserves the comparator's
\emph{polynomial stationarity exponent} up to logarithmic factors and
finite-horizon constants. The guarantee concerns rate preservation, not
convergence of the CARP-KL/Oracle-KL ratio to one.

For FLAT, numerical floors make absolute stationarity and seed dispersion
essential complements to paired ratios. Appendix~\ref{app:experiments}
reports the complete paired-gap trajectory together with these diagnostics.
The finite-budget experiments characterize the behavior of the complete
adaptive procedure; the KL-dependent asymptotic exponent is established
by the theoretical analysis.

\paragraph{Portfolio size and resource reallocation.}
Corollary~\ref{cor:reduced} predicts a structural change in resource
allocation when a common conservative initial-gap bound removes one portfolio
coordinate. We probe this consequence directly by comparing the full
three-dimensional CARP-KL portfolio with its reduced two-dimensional
counterpart on
$
  \log_2T\in\{15,17,19,20\},
$
using 20 evaluation seeds and otherwise matched accounting.

At $T=2^{20}$, the portfolio size decreases from $81$ to $27$ candidates,
leaving substantially more primitive calls for both optimization and
certification of each candidate. Median stationarity is lower for the reduced
portfolio on P2--P4 and higher on FLAT at this budget; across the four tested
budgets, however, the effect varies by problem.
Proposition~\ref{prop:reduced-tradeoff} characterizes the two competing
effects: reducing the portfolio improves both the per-candidate optimization
budget and the certification resolution, while replacing the independently
adapted gap scale by the shared bound incurs the explicit penalty
$\Gamma_\Delta^{\bar\theta}$. In particular, the proposition gives a
sufficient condition under which the reduced-portfolio bound for the same
comparator is no larger than its full-portfolio counterpart. The experiments
illustrate the complementary finite-budget regime, in which the resource
gain need not dominate the scale-mismatch cost uniformly across problems.

\paragraph{Where is the remaining finite-budget loss?}
A separate mechanism diagnostic at $T=2^{22}$ distinguishes candidate
generation from candidate recovery. Strong candidates are present in the
CARP-KL portfolio, but the finally selected candidate can remain far from the
best realized rank, while certification consumes roughly half of the
primitive-call budget. Replay analyses, selector ablations, and heavy-tail
stress tests reinforce the same finite-budget picture: in the tested
construction, certified selection is a principal \emph{observed} bottleneck
once strong candidates are available.

Proposition~\ref{prop:selection-resolution} formalizes this distinction.
Under uniform certification, the selected point competes directly with the
best realized candidate up to a resolution term of order
$
  L\sqrt{K_T\log(K_T/\delta)/T}.
$
Moreover, on the simultaneous certificate event, candidates separated from
the realized best by more than the corresponding certificate resolution
cannot be selected. Thus the presence of a strong candidate does not by
itself guarantee its recovery when candidate gaps lie at or below the
finite-budget resolution scale. This explains the observed separation between
candidate quality and returned performance and motivates allocating
certification effort more selectively across candidates.
Appendix~\ref{app:experiments} reports the full evidence, including negative
and nonmonotone outcomes.
% END sections/07_experiments

% BEGIN sections/08_limitations
\section{Limitations and Conclusion}
\label{sec:limitations}

CARP-KL provides same-oracle, budget-accounted adaptation to unknown and
potentially nonunique KL descriptions in stochastic weakly-convex
optimization. It competes with valid affordable descriptions and preserves
their known-description polynomial Moreau-stationarity rates up to logarithmic
factors. With a common conservative initial-gap bound, the portfolio reduces
from $O((\log T)^3)$ to $O((\log T)^2)$ candidates, with an explicit
scale-mismatch penalty.

The guarantee is horizon-aware and projected, under known $(\rho,L,T)$,
declared parameter ranges, a nonempty affordable comparator class, and
common-region KL validity. Within this scope, CARP-KL preserves the
known-description polynomial exponent while making competition among unknown
KL descriptions observable and executable under the same primitive-oracle
budget. The experiments show that strong candidates can coexist with imperfect
recovery and that smaller portfolios reallocate budget without uniformly
improving the returned point. These findings motivate adaptive certification
and portfolio allocation. More broadly, when the relevant geometry is
nonunique and candidate quality is not directly observable,
\emph{certified selection} is itself part of the adaptation problem.
% END sections/08_limitations

\bibliography{references}

\clearpage
\appendix
\section*{Supplementary Analysis}

% BEGIN appendices/app_a_related_notation
\section{Extended Related Work and Notation}
\label{app:related}

This appendix develops the positioning summarized in
Section~\ref{sec:related} and fixes the notation used throughout the
supplementary analysis. The key distinction is between robustness or
adaptation of a prescribed optimization schedule and observable competition
among multiple candidate KL descriptions. Theorem~\ref{thm:master} addresses
the latter under a single primitive-oracle budget: candidate generation,
fresh stationarity certification, and final selection all use the same
primitive stochastic-subgradient oracle, with every oracle call charged to
the common horizon.

\subsection{Extended positioning}

\begin{table}[!htp]
\caption{Methodological precedents and their relation to the present result.
The comparison records the adaptation mechanism, information model, and
performance target relevant to Theorem~\ref{thm:master}.}
\label{tab:extended-related}
\centering

\resizebox{0.63\linewidth}{!}{%
\begin{tabular}{@{}p{0.17\textwidth}p{0.20\textwidth}p{0.26\textwidth}p{0.48\textwidth}@{}}
\toprule
Work
& Adaptation device
& Setting and reported target
& Relation to the present theorem \\
\midrule

\citet{roulet2017sharpness}
& search over unknown sharpness parameters
& convex first-order optimization; objective error
& establishes a precedent for adapting restart schemes to unknown sharpness;
CARP-KL instead targets Moreau stationarity in stochastic weakly-convex
optimization and certifies competing candidate outputs through the primitive
oracle
\\
\midrule

\citet{adcock2025restarts}
& search over multiple sharpness parameters
& convex optimization under unknown sharpness descriptions
& demonstrates joint adaptation to geometric parameters in a convex setting;
CARP-KL addresses KL-description competition under stochastic weak convexity
and an explicit certification budget
\\
\midrule

\citet{renegar2022restart}
& restart based on observable progress
& convex optimization; objective-based progress
& provides an observable restart mechanism when objective-based progress is
available; CARP-KL constructs a stationarity certificate from the primitive
stochastic-subgradient oracle instead
\\
\midrule

\citet{fercoq2019adaptive}
& adaptive or variable-period restart
& convex smooth/composite optimization with observable progress measures
& develops adaptive restart from observable optimization progress; the
present setting replaces such progress information by fresh same-oracle
Moreau-stationarity certification
\\
\midrule

\citet{xu2017adaptive}
& search over unknown growth parameters
& convex finite-sum optimization under error-bound structure
& provides an unknown-geometry adaptation precedent; CARP-KL treats a joint
KL description and selects among the resulting stochastic candidate outputs
\\
\midrule

\citet{yang2018rsg}
& restarted stochastic subgradient stages
& convex stochastic subgradient optimization under local error bounds
& establishes stochastic restart acceleration under growth structure;
CARP-KL extends the adaptation question to weakly-convex Moreau stationarity
with unknown KL descriptions and certified output selection
\\
\midrule

\citet{fatkhullin2022sharp}
& KL/sharpness-based stochastic acceleration
& stochastic optimization under additional regularity or oracle structure
& demonstrates KL-dependent stochastic acceleration; the present theorem
focuses on adaptation and observable candidate competition under the stated
primitive-oracle model
\\
\midrule

\citet{xie2026adaptive}
& online adaptation of the proximal parameter
& stochastic weakly-convex optimization with unknown weak-convexity
parameter; Moreau-envelope stationarity
& adapts to an unknown weak-convexity scale using stochastic candidate and
function-difference information; CARP-KL instead assumes $\rho$ known and
adapts over unknown KL descriptions using only the primitive
stochastic-subgradient oracle for both optimization and certification
\\
\midrule

\citet{li2026optimal}
& deterministic oracle-complexity lower bound
& Lipschitz weakly-convex optimization; Moreau-envelope stationarity
& establishes the generic $\varepsilon^{-4}$ deterministic complexity under
an oracle returning function values and full subdifferentials; it provides
context for generic weakly-convex stationarity rather than a lower bound for
KL adaptation or certified selection
\\
\midrule

\citet{atenas2026proximal}
& proximal stochastic subgradient with adaptive sampling and line search
& nonsmooth nonconvex stochastic optimization under KL geometry; trajectory
and function-value convergence
& exploits KL geometry to obtain whole-trajectory convergence and explicit
rates under a sample-average and line-search framework; CARP-KL studies
finite-horizon competition among unknown KL descriptions under a fixed
primitive-oracle budget
\\
\midrule

\citet{jiang2026nonsmooth}
& Bregman proximal and model-based stochastic minimization
& nonsmooth relatively weakly-convex composite optimization; stationarity
& develops deterministic and stochastic methods for a broader relative
weak-convexity model, including generic $\varepsilon^{-4}$ stochastic
complexity; the present theorem addresses KL-description adaptation and
same-oracle candidate certification
\\
\midrule

\citet{elhishi2026restart}
& KL-dependent restart with robustness to schedule misspecification
& nonsmooth stochastic weakly-convex optimization; KL-accelerated convergence
& establishes KL-dependent restart acceleration and near-adaptivity through
robustness to substantial schedule misspecification; CARP-KL addresses the
complementary problem of observable competition among unknown and potentially
nonunique KL descriptions, with training and certification charged to one
primitive-oracle budget
\\
\midrule

\citet{carmon2022making}
& certification for parameter-free optimization
& stochastic optimization with unknown problem parameters
& establishes certification as a mechanism for parameter-free adaptation;
CARP-KL develops the corresponding primitive-oracle certificate for
Moreau stationarity and couples it to a validity-preserving KL portfolio
\\
\midrule

\citet{attia2024free}
& generic hyperparameter search
& parameter-free stochastic optimization
& provides a general search paradigm for unknown algorithmic parameters;
CARP-KL specializes the adaptation mechanism to one-sided KL discretization
and same-oracle Moreau-stationarity selection
\\
\midrule

\citet{carmon2024price}
& lower bounds for adaptation
& stochastic convex optimization
& establishes that adaptation can carry nonconstant cost in other stochastic
optimization settings and motivates explicit accounting of the portfolio and
certification overhead in the present theorem
\\

\bottomrule
\end{tabular}%
}

\end{table}

\paragraph{Robust schedules and observable competition.}
Robustness to schedule misspecification and observable competition solve
distinct adaptation problems. A KL-dependent restart rule can remain effective
under inaccurate geometric parameters without identifying which of several
random candidate outputs should ultimately be returned. CARP-KL addresses this
second problem directly. Multiple KL descriptions generate candidate
trajectories, while neither the objective gap nor
\[
  S(x)=\|\nabla f_{1/(2\rho)}(x)\|
\]
is supplied by the primitive oracle. Theorem~\ref{thm:master} therefore analyzes the complete interface
from candidate generation through fresh certification to final selection,
with all primitive calls charged to a single horizon. This complements
robustness-based KL restart results by making competition among candidate
descriptions observable and explicitly budget-accounted.

\paragraph{Adaptation over nonunique KL descriptions.}
The comparator is a KL description rather than an estimated latent parameter.
An objective may satisfy the required KL inequality for several triples
\[
  \kappa=(\theta,c,\Delta)
\]
with different exponents, constants, stage scales, and finite-horizon
guarantees. Accordingly, Theorem~\ref{thm:master} competes over the complete
class of valid affordable descriptions. This formulation naturally captures
finite-horizon crossings: two valid descriptions can exchange their ordering
as the horizon changes because both their polynomial exponents and their
finite-budget constants enter the restart guarantee.

\paragraph{Oracle and performance-metric separation.}
The information model is part of the comparison. The primitive
oracle~\eqref{eq:oracle} returns a stochastic subgradient with a uniform
second-moment bound. Models that additionally expose function values,
function differences, finite-sum components, proximal maps, or stronger
gradient information can support different adaptation and validation
mechanisms. Likewise, objective error, distance to a solution set, gradient
mappings, function-value convergence, and Moreau stationarity encode
different performance criteria. The comparisons above therefore align each
method with its stated oracle and target before identifying its relation to
Theorem~\ref{thm:master}.

\paragraph{Sharpness terminology and exponent conventions.}
KL, PL or gradient-domination, H\"olderian error-bound, and sharpness
conditions are expressed through different quantities and exponent
conventions across the literature. In this manuscript, the geometric
description is fixed by
\[
  S(x)\geq
  c\bigl(F(x)-f^\star\bigr)^\theta,
  \qquad
  \theta\in[1/2,1],
\]
and the resulting restart exponent is
\[
  \nu_\theta
  =
  \frac{\theta}{4\theta-1}.
\]
All rate comparisons are interpreted through this definition together with
the corresponding oracle model and performance metric.

\subsection{Notation and proof conventions}

Throughout,
\[
  F=f_{1/(2\rho)},\qquad
  S(x)=\|\nabla F(x)\|,\qquad
  \Delta_0=F(x_0)-f^\star.
\]
A continuous KL description is written
\[
  \kappa=(\theta,c,\Delta).
\]
Its one-sided full-grid neighbor is
\[
  \bar\kappa
  =
  (\bar\theta,\underline c,\bar\Delta),
\]
with
\[
  \bar\theta\geq\theta,\qquad
  \underline c\leq c,\qquad
  \bar\Delta\geq\Delta.
\]
The full product grid is denoted by
\[
  \cG_T,
  \qquad
  K_T=|\cG_T|.
\]

For a primitive-call cap $u$, the known-description restart guarantee is denoted by $\cR_u(\kappa),$
and
\[
  \nu_\theta=\frac{\theta}{4\theta-1}
\]
denotes its polynomial exponent. The class
$\Kaff$ contains the valid continuous descriptions whose one-sided neighbors
are affordable under the per-candidate training allocation. Appendix
\ref{app:restart} establishes the known-description cap guarantee, while
Appendix~\ref{app:grid} transfers it to the corresponding recursion-valid
grid neighbor.

For a portfolio of $K$ candidates, $n$ and $m$ denote the per-candidate
training and certification allocations. Certification uses $B_K$ independent
blocks of length $q_T$, and $r_{q_T}$ denotes the simultaneous confidence
radius. For candidate $j$, $x_j$ denotes its trained output and $\widehat S_j$
its fresh same-oracle stationarity certificate. The selected output is
\[
  \widehat x_T=x_{\widehat j}.
\]
Candidate training runs may be dependent or share randomness. The
concentration argument requires freshness of the certification samples
conditional on the trained portfolio.

For the reduced construction, a common conservative gap bound
$\Delta_+$ satisfies
\[
  \Delta_+\geq\Delta,
\]
and
\[
  \Gamma_\Delta
  =
  \frac{\Delta_+}{\Delta}
\]
denotes the resulting scale mismatch. The reduced neighbor is written
\[
  \widetilde\kappa
  =
  (\bar\theta,\underline c,\Delta_+),
\]
and the corresponding two-dimensional grid and its cardinality are
\[
  \cG_T^{\rm red},
  \qquad
  K_T^{\rm red}=|\cG_T^{\rm red}|.
\]

These conventions separate the three objects used by the proof:
the valid continuous description $\kappa$, its executable one-sided grid
neighbor, and the random candidate output generated by that neighbor. This
separation is maintained throughout the restart, discretization,
certification, and portfolio-selection arguments.
% END appendices/app_a_related_notation

% BEGIN appendices/app_b_restart
\section{Known-Description Restart Analysis}
\label{app:restart}

This appendix proves Proposition~\ref{prop:restart} and establishes the
known-description restart interface used throughout the portfolio analysis.
The construction follows a deterministic upper-gap trajectory, implements
each restart stage with an integer number of primitive oracle calls, and
converts any external per-candidate call cap into an explicit
Moreau-stationarity guarantee. The resulting bound depends only on the KL
description and the allocated primitive-oracle budget; neither objective
values nor realized envelope gaps are required during execution.

Throughout this appendix, fix a valid description
\[
  \kappa=(\theta,c,\Delta),\qquad
  \theta\in[1/2,1],\quad c>0,\quad
  0<\Delta_0\leq\Delta\leq1.
\]
The same analysis applies when $\Delta_0=0$ and a positive upper bound
$\Delta$ is used. If $\Delta_0=\Delta=0$, then $x_0$ minimizes the smooth
envelope $F$, so $S(x_0)=0$ and the stationarity target is already attained.

Let $\mathcal F_k$ denote the sigma-field generated by all randomness before
stage $k$. Conditional on $\mathcal F_k$, the stage starts from $x_k$ and runs
$t_k$ projected stochastic-subgradient calls with constant step
\[
  \alpha_k=\frac{\beta_k}{\sqrt{t_k+1}}.
\]
As in the standard weakly-convex stochastic-subgradient construction
\citep[Corollaries 2.2 and 2.6]{davis2019stochastic}, the output $x_{k+1}$ is
a randomly selected inner iterate with probabilities proportional to the
step sizes. This randomization is part of the base method and is retained by
every portfolio candidate.

\subsection{Conditional one-stage inequality}
\label{app:restart:onestage}

The restart analysis begins from the projected stochastic-subgradient
interface. The projected setting permits arbitrary $\beta_k>0$, which is
exactly the flexibility required by the deterministic upper-gap schedule.

\begin{lemma}[Conditional projected-SGD bound]
\label{lem:base-sgd}
Suppose \eqref{eq:oracle} holds and $x_k$ is $\mathcal F_k$-measurable. For
any integer $t_k\geq1$ and any $\mathcal F_k$-measurable $\beta_k>0$, the
stage output obeys
\begin{equation}
 \E\!\left[S(x_{k+1})^2\mid\mathcal F_k\right]
 \leq
 \frac{2(F(x_k)-f^\star)}{\beta_k\sqrt{t_k+1}}
 +\frac{2\rho L^2\beta_k}{\sqrt{t_k+1}}.
 \label{eq:base-stage}
\end{equation}
\end{lemma}

\begin{proof}
Conditional on $\mathcal F_k$, the starting point and step are fixed.
Applying the projected version of the standard weakly-convex stochastic
subgradient bound with envelope parameter $1/(2\rho)$ gives
\eqref{eq:base-stage}. The oracle samples inside the stage are fresh, so their
conditional law satisfies \eqref{eq:oracle}. The projected result permits
arbitrary positive $\beta_k$, completing the conditional application.
\end{proof}

The restart schedule uses a deterministic upper bound on the expected
envelope gap. This converts the conditional inequality into a stage guarantee
whose parameters are fixed before the stage is executed.

\begin{lemma}[One stage from a deterministic gap target]
\label{lem:deterministic-target}
Let $D_k>0$ be deterministic and suppose
\begin{equation}
  \E[F(x_k)-f^\star]\leq D_k.
  \label{eq:target-hypothesis}
\end{equation}
Choose
\begin{equation}
  \beta_k=\frac{\sqrt{D_k}}{\sqrt\rho L}.
  \label{eq:beta-target}
\end{equation}
Then
\begin{equation}
  \E S(x_{k+1})^2
  \leq
  \frac{4\sqrt\rho L\sqrt{D_k}}{\sqrt{t_k+1}}.
  \label{eq:one-stage-target}
\end{equation}
\end{lemma}

\begin{proof}
Taking total expectation in \eqref{eq:base-stage} and using
\eqref{eq:target-hypothesis}--\eqref{eq:beta-target} gives
\[
 \frac{\E[F(x_k)-f^\star]}{\beta_k}
 \leq
 \sqrt\rho L\sqrt{D_k},
 \qquad
 \rho L^2\beta_k
 =
 \sqrt\rho L\sqrt{D_k}.
\]
Substitution proves \eqref{eq:one-stage-target}. In particular, the stage
requires neither $F(x_k)$ nor $f^\star$ as algorithmic input.
\end{proof}

\subsection{Pointwise KL geometry and the stochastic recursion}
\label{app:restart:klconversion}

The pointwise KL inequality induces the expectation-level relation required by
the restart induction. This step uses both $\theta\geq1/2$ and the common-region
condition.

\begin{lemma}[Stochastic KL conversion]
\label{lem:stochastic-kl}
Let $X$ be a random point satisfying $X\in\mathcal V$ almost surely.
Under Assumption~\ref{ass:klc},
\begin{equation}
 \bigl(\E S(X)^2\bigr)^{1/2}
 \geq
 c\bigl(\E[F(X)-f^\star]\bigr)^\theta.
 \label{eq:stochastic-kl}
\end{equation}
Equivalently,
\begin{equation}
 \E[F(X)-f^\star]
 \leq
 c^{-1/\theta}
 \bigl(\E S(X)^2\bigr)^{1/(2\theta)}.
 \label{eq:stochastic-kl-inverse}
\end{equation}
\end{lemma}

\begin{proof}
Pointwise KL and nonnegativity give
\[
 \E S(X)^2
 \geq
 c^2\E(F(X)-f^\star)^{2\theta}.
\]
Since $2\theta\geq1$, the map $z\mapsto z^{2\theta}$ is convex on
$\mathbb R_+$. Jensen's inequality therefore yields
\[
 \E(F(X)-f^\star)^{2\theta}
 \geq
 \bigl(\E[F(X)-f^\star]\bigr)^{2\theta}.
\]
Taking square roots proves \eqref{eq:stochastic-kl}; rearranging gives
\eqref{eq:stochastic-kl-inverse}.
\end{proof}

\begin{remark}[Common-region interface]
\label{rem:common-region-restart}
Assumption~\ref{ass:klc} places the random stage outputs in a region on which
the pointwise KL inequality holds almost surely. Lemma~\ref{lem:stochastic-kl}
then converts this pointwise geometry into the stochastic inequality consumed
by the restart recursion. Region membership is therefore an analytical
property of the comparator trajectory rather than an observable quantity
required by the algorithm.
\end{remark}

\subsection{Upper-gap induction with integer stages}
\label{app:restart:induction}

Set
\begin{equation}
 \gamma_\theta=\frac{1}{4\theta},\qquad
 p_\theta=4\theta-1,\qquad
 a_\theta=p_\theta\gamma_\theta
 =1-\frac{1}{4\theta},
 \label{eq:restart-exponents}
\end{equation}
and define
\begin{align}
 D_k&=\Delta e^{-\gamma_\theta k},
 &\beta_k&=\frac{\sqrt{D_k}}{\sqrt\rho L},
 \label{eq:restart-targets}\\
 A_\kappa&=16e\,c^{-4}\rho L^2\Delta^{-p_\theta},
 &t_k&=\left\lceil A_\kappa e^{a_\theta k}\right\rceil.
 \label{eq:restart-schedule-app}
\end{align}
The factor $e$ in $A_\kappa$ supplies the one-stage slack required to advance
the deterministic target by one restart level. Since $A_\kappa>0$, every
$t_k\geq1$. The quantities $D_k$, $\beta_k$, and $t_k$ are deterministic
functions of the candidate description and the stage index and are therefore
fixed independently of the realized objective gap.

\begin{lemma}[Upper-gap restart induction]
\label{lem:restart-induction}
For every integer $k\geq0$,
\begin{equation}
 \E[F(x_k)-f^\star]\leq D_k.
 \label{eq:gap-stage-app}
\end{equation}
Moreover, for every $k\geq1$,
\begin{equation}
 \bigl(\E S(x_k)^2\bigr)^{1/2}
 \leq
 c\Delta^\theta e^{-\theta\gamma_\theta k}.
 \label{eq:stat-stage-app}
\end{equation}
\end{lemma}

\begin{proof}
The base case follows from
\[
 F(x_0)-f^\star
 =
 \Delta_0
 \leq
 \Delta
 =
 D_0.
\]
Suppose \eqref{eq:gap-stage-app} holds at stage $k$.
Lemma~\ref{lem:deterministic-target} gives
\begin{equation}
 \E S(x_{k+1})^2
 \leq
 \frac{4\sqrt\rho L\Delta^{1/2}e^{-\gamma_\theta k/2}}
 {\sqrt{t_k+1}}.
 \label{eq:stage-before-length}
\end{equation}
% By \eqref{eq:restart-schedule-app},
% \begin{align*}
%  t_k+1
%  &\geq
%  A_\kappa e^{a_\theta k}\\
%  &=
%  16c^{-4}\rho L^2\Delta^{1-4\theta}
%  e^{4\theta\gamma_\theta}
%  e^{(4\theta-1)\gamma_\theta k}.
% \end{align*}
By \eqref{eq:restart-schedule-app},
\[
  t_k + 1 \ge A_\kappa e^{a_\theta k} = 16c^{-4}\rho L^2\Delta^{1-4\theta} e^{4\theta\gamma_\theta} e^{(4\theta-1)\gamma_\theta k}.
\]
Substituting into \eqref{eq:stage-before-length} gives
\begin{equation}
 \E S(x_{k+1})^2
 \leq
 c^2\Delta^{2\theta}
 e^{-2\theta\gamma_\theta(k+1)}.
 \label{eq:stage-stationarity-square}
\end{equation}
This proves \eqref{eq:stat-stage-app} at index $k+1$.
Lemma~\ref{lem:stochastic-kl} then yields
% \begin{align*}
%  \E[F(x_{k+1})-f^\star]
%  &\leq
%  c^{-1/\theta}
%  \left(
%  c^2\Delta^{2\theta}
%  e^{-2\theta\gamma_\theta(k+1)}
%  \right)^{1/(2\theta)}\\
%  &=
%  \Delta e^{-\gamma_\theta(k+1)}
%  =
%  D_{k+1}.
% \end{align*}
\[
  \mathbb{E}[F(x_{k+1})-f^\star]
  \le c^{-1/\theta} \left( c^2\Delta^{2\theta} e^{-2\theta\gamma_\theta(k+1)} \right)^{1/(2\theta)}
  = \Delta e^{-\gamma_\theta(k+1)}
  = D_{k+1}.
\]
The induction follows.
\end{proof}

\begin{remark}[Observable execution]
The schedule fixes $D_k$, $\beta_k$, and $t_k$ before stage $k$ is executed.
The expectation bound \eqref{eq:gap-stage-app} is used only in the analysis.
Consequently, the restart procedure requires no envelope-gap evaluation,
objective optimum, Moreau gradient, or realized stage-error test.
\end{remark}

\subsection{Completed-stage complexity}
\label{app:restart:completed}

We first express the guarantee in terms of the exact cost of completed restart
stages. Define
\[
 T_r(\kappa)
 =
 \sum_{k=0}^{r-1}t_k,\qquad
 T_0(\kappa)=0,
 \qquad
 b_\theta=e^{a_\theta}-1.
\]
Since $e^{a_\theta k}\geq1$ and $\lceil z\rceil\leq z+1$,
\begin{equation}
 A_\kappa e^{a_\theta k}
 \leq
 t_k
 \leq
 (1+A_\kappa)e^{a_\theta k}.
 \label{eq:integer-stage-sandwich}
\end{equation}
Hence, for every $r\geq1$,
\begin{equation}
 T_r(\kappa)
 \leq
 (1+A_\kappa)
 \frac{e^{a_\theta r}-1}{b_\theta}.
 \label{eq:completed-cost-upper}
\end{equation}
Combining this inequality with Lemma~\ref{lem:restart-induction} gives
\begin{align}
 \E[F(x_r)-f^\star]
 &\leq
 \Delta
 \left[
 1+\frac{b_\theta T_r(\kappa)}{1+A_\kappa}
 \right]^{-1/p_\theta},
 \label{eq:completed-gap}\\
 \bigl(\E S(x_r)^2\bigr)^{1/2}
 &\leq
 c\Delta^\theta
 \left[
 1+\frac{b_\theta T_r(\kappa)}{1+A_\kappa}
 \right]^{-\nu_\theta},
 \label{eq:completed-stat}
\end{align}
where
\[
 \nu_\theta=\frac{\theta}{p_\theta}.
\]
These bounds are indexed by the exact cost of the completed stages and
therefore incur no last-stage conversion factor.

\subsection{External-cap guarantee}
\label{app:restart:cap}

The portfolio assigns each candidate an external integer training cap. For
$u\geq0$, define
\begin{equation}
 N(u,\kappa)
 =
 \max\{r\geq0:T_r(\kappa)\leq u\}.
 \label{eq:completed-stage-count}
\end{equation}
The implementation executes complete stages only and returns the output of
stage $N(u,\kappa)$. Any calls remaining after the last completed stage stay
within the allocated cap.

\begin{lemma}[External-cap restart bound]
\label{lem:external-cap}
If
\begin{equation}
 t_0=\lceil A_\kappa\rceil\leq u,
 \label{eq:first-stage-affordable}
\end{equation}
then $N=N(u,\kappa)\geq1$ and
\begin{align}
 \E[F(x_N)-f^\star]
 &<
 e^{\gamma_\theta}\Delta
 \left[
 1+\frac{b_\theta u}{1+A_\kappa}
 \right]^{-1/p_\theta},
 \label{eq:external-gap}\\
 \bigl(\E S(x_N)^2\bigr)^{1/2}
 &<
 e^{1/4}c\Delta^\theta
 \left[
 1+\frac{b_\theta u}{1+A_\kappa}
 \right]^{-\nu_\theta}
 =
 \cR_u(\kappa).
 \label{eq:external-stat}
\end{align}
The corresponding non-strict inequalities hold as well and are used by the
main theorem.
\end{lemma}

\begin{proof}
Condition \eqref{eq:first-stage-affordable} gives $N\geq1$.
By maximality in \eqref{eq:completed-stage-count},
\begin{equation}
 u<T_{N+1}(\kappa).
 \label{eq:cap-maximality}
\end{equation}
Applying \eqref{eq:completed-cost-upper} with $r=N+1$ gives
\[
 u<
 (1+A_\kappa)
 \frac{e^{a_\theta(N+1)}-1}{b_\theta},
\]
and hence
\begin{equation}
 e^{a_\theta(N+1)}
 >
 Q_u(\kappa)
 :=
 1+\frac{b_\theta u}{1+A_\kappa}.
 \label{eq:cap-Q}
\end{equation}
Since $a_\theta=p_\theta\gamma_\theta$,
\[
 e^{-\gamma_\theta N}
 =
 e^{\gamma_\theta}e^{-\gamma_\theta(N+1)}
 <
 e^{\gamma_\theta}Q_u(\kappa)^{-1/p_\theta}.
\]
Combining this inequality with \eqref{eq:gap-stage-app} proves
\eqref{eq:external-gap}. Similarly,
\[
 e^{-\theta\gamma_\theta N}
 <
 e^{\theta\gamma_\theta}
 Q_u(\kappa)^{-\theta\gamma_\theta/a_\theta}.
\]
Finally,
\[
 \theta\gamma_\theta=\frac14,
 \qquad
 \frac{\theta\gamma_\theta}{a_\theta}
 =
 \frac{\theta}{4\theta-1}
 =
 \nu_\theta.
\]
Equation \eqref{eq:stat-stage-app} now gives
\eqref{eq:external-stat}.
\end{proof}

Lemma~\ref{lem:external-cap} establishes Proposition~\ref{prop:restart} in the
form required by the portfolio analysis, including the $e^{1/4}$ cap-conversion
factor and the ceiling-safe denominator $1+A_\kappa$.

\subsection{Conservative gap bounds}
\label{app:restart:misspecification}

The restart construction remains valid under any deterministic upper bound on
the initial envelope gap. This property supports both dyadic gap
discretization and the reduced portfolio with a shared conservative gap bound.

Fix $(\theta,c)$ and let
\[
 \Delta_0\leq\Delta\leq\widetilde\Delta\leq1,
 \qquad
 \Gamma=\frac{\widetilde\Delta}{\Delta}\geq1.
\]
Then
\begin{equation}
 \frac{A_{(\theta,c,\widetilde\Delta)}}
      {A_{(\theta,c,\Delta)}}
 =
 \Gamma^{-p_\theta}
 \leq1.
 \label{eq:gap-scale-A-ratio}
\end{equation}
Thus increasing the valid gap upper bound weakly decreases the first-stage
scale. For every common cap $u$,
\[
 Q_u(\theta,c,\widetilde\Delta)
 =
 1+
 \frac{b_\theta u}
 {1+A_{(\theta,c,\widetilde\Delta)}}
 \geq
 Q_u(\theta,c,\Delta).
\]
It follows that affordability transfers from $\Delta$ to
$\widetilde\Delta$, and
\begin{align}
 \frac{\text{gap bound at }\widetilde\Delta}
      {\text{gap bound at }\Delta}
 &\leq
 \Gamma,
 \\
 \frac{\cR_u(\theta,c,\widetilde\Delta)}
      {\cR_u(\theta,c,\Delta)}
 &\leq
 \Gamma^\theta.
 \label{eq:gap-misspec-ratio}
\end{align}
In particular, a dyadic upper neighbor
$\Delta\leq\widetilde\Delta<2\Delta$ incurs at most a factor $2$ in the
envelope-gap bound and at most $2^\theta\leq2$ in RMS stationarity.

For a shared conservative bound $\Delta_+$, define
\[
 \Gamma_\Delta=\frac{\Delta_+}{\Delta}.
\]
Then
\begin{equation}
 \cR_u(\theta,c,\Delta_+)
 \leq
 \Gamma_\Delta^\theta
 \cR_u(\theta,c,\Delta).
 \label{eq:shared-gap-rate}
\end{equation}
Appendix~\ref{app:grid} combines this scale relation with one-sided rounding
of $\theta$ and $c$. The resulting reduced construction removes the gap
coordinate from the portfolio and quantifies its complete comparison through
the factor $\Gamma_\Delta^{\bar\theta}$.

\subsection{Boundary cases and scope}
\label{app:restart:corners}

\paragraph{Endpoint $\theta=1/2$.}
Here
\[
 p_\theta=1,\qquad
 \gamma_\theta=\frac12,\qquad
 a_\theta=\frac12,\qquad
 \nu_\theta=\frac12.
\]
All quantities are well defined, and \eqref{eq:external-stat} gives the
$u^{-1/2}$ stationarity exponent.

\paragraph{Endpoint $\theta=1$.}
Here
\[
 p_\theta=3,\qquad
 \gamma_\theta=\frac14,\qquad
 a_\theta=\frac34,\qquad
 \nu_\theta=\frac13.
\]
In particular, $b_\theta=e^{3/4}-1>0$, and
\eqref{eq:external-stat} gives the $u^{-1/3}$ endpoint exponent directly.

\paragraph{First-stage affordability.}
The cap guarantee applies once
\[
 \lceil A_\kappa\rceil\leq u.
\]
If this condition fails, the cap contains no complete restart stage and the
implementation returns $x_0$. The master theorem therefore compares against
the descriptions whose first stage fits the allocated training cap.

\paragraph{Unbounded feasible sets.}
The argument never invokes $\operatorname{diam}(X)$. It requires a closed
convex feasible set, the uniform second-moment oracle bound, and the
common-region KL condition along the comparator trajectory.

\paragraph{Description validity.}
The restart schedule is executable for any declared parameter triple. Its
theoretical guarantee is attached to valid descriptions, while
Appendix~\ref{app:grid} establishes the corresponding recursion-valid
interface for their one-sided grid neighbors.

\paragraph{Allocated and consumed calls.}
The cap $u$ is the quantity charged to the global portfolio budget. A
candidate may consume fewer calls when the remaining allocation cannot fit
the next complete stage. The external-cap bound already incorporates this
integer-stage effect, so no redistribution of unused calls is required by the
theory.

\subsection{Interface to the portfolio theorem}
\label{app:restart:contract}

For every valid description $\kappa$ whose first stage is affordable under an
integer cap $u$, define
\[
 x^{(\kappa,u)}
 :=
 x_{N(u,\kappa)}.
\]
The analysis above establishes
\begin{equation}
 \bigl(\E S(x^{(\kappa,u)})^2\bigr)^{1/2}
 \leq
 \cR_u(\kappa),
 \label{eq:restart-export-rms}
\end{equation}
where
\begin{equation}
 \cR_u(\kappa)
 =
 e^{1/4}c\Delta^\theta
 \left[
 1+
 \frac{(e^{1-1/(4\theta)}-1)u}
 {1+16e\,c^{-4}\rho L^2\Delta^{-(4\theta-1)}}
 \right]^{-\theta/(4\theta-1)}.
 \label{eq:restart-export-rate}
\end{equation}
Consequently,
\begin{equation}
 \E S(x^{(\kappa,u)})
 \leq
 \bigl(\E S(x^{(\kappa,u)})^2\bigr)^{1/2}
 \leq
 \cR_u(\kappa).
 \label{eq:restart-export-mean}
\end{equation}

Equations~\eqref{eq:restart-export-rms}--\eqref{eq:restart-export-mean} are
the complete known-description interface used by
Appendices~\ref{app:grid} and~\ref{app:master}. They provide an explicit
finite-cap stationarity guarantee from the primitive stochastic-subgradient
oracle, require no objective-value or $f^\star$ access during execution, and
remain valid on unbounded feasible sets under the stated oracle and
common-region assumptions.
% END appendices/app_b_restart

% BEGIN appendices/app_c_grid
\section{Joint KL Discretization}
\label{app:grid}

This appendix proves Lemma~\ref{lem:grid} and establishes the discretization
interface used by the master theorem. The construction rounds the KL exponent
upward, the KL constant downward, and the initial-gap upper bound upward. With
the horizon-dependent exponent mesh defined below, every valid continuous
description in the declared ranges admits a grid neighbor whose restart
trajectory inherits the known-description cap guarantee up to a uniform
constant factor. The resulting portfolio has $O((\log T)^3)$ candidates.

The proof separates geometric validity from schedule approximation. The
original KL description remains the pointwise geometric certificate on the
common region, while the one-sided grid neighbor specifies the executable
restart schedule. The one-sided parameter inequalities ensure that the
stationarity target produced by the rounded schedule is strong enough to close
the expected-gap induction through the original stochastic KL inequality.
No realized-gap bound is required.

Fix a horizon $T\geq1$ and declared ranges
\begin{equation}
 \theta\in[1/2,1],\qquad
 c\in[c_{\min,T},c_{\max,T}],\qquad
 \Delta\in[\Delta_{\min,T},1],
 \label{eq:grid-ranges}
\end{equation}
where
\begin{equation}
 \frac{c_{\max,T}}{c_{\min,T}}\leq T^{r_c},
 \qquad
 \Delta_{\min,T}=T^{-r_\Delta},
 \qquad
 r_c,r_\Delta\geq0.
 \label{eq:grid-polynomial-ranges}
\end{equation}
The constants $r_c,r_\Delta$ are declared and independent of the unknown KL
description.

\subsection{Expectation-level one-sided nesting}
\label{app:grid:nesting}

Let
\[
 Z(x)=F(x)-f^\star\geq0.
\]
Assumption~\ref{ass:klc} supplies the original pointwise KL inequality on the
common region. The normalization enters through the deterministic expectation
targets used by the restart recursion rather than through a pointwise bound on
realized envelope gaps.

\begin{lemma}[One-sided stochastic KL nesting]
\label{lem:grid-nesting}
Suppose $(\theta,c)$ satisfies the pointwise KL inequality on $\mathcal V$.
Let $X$ be a random point satisfying $X\in\mathcal V$ almost surely and
\[
 \E Z(X)\leq1.
\]
For any $\bar\theta\geq\theta$ and
$0<\underline c\leq c$,
\begin{equation}
 \bigl(\E S(X)^2\bigr)^{1/2}
 \geq
 \underline c\bigl(\E Z(X)\bigr)^{\bar\theta}.
 \label{eq:grid-stochastic-nesting}
\end{equation}
\end{lemma}

\begin{proof}
Lemma~\ref{lem:stochastic-kl}, applied to the original pointwise description,
gives
\[
 \bigl(\E S(X)^2\bigr)^{1/2}
 \geq
 c\bigl(\E Z(X)\bigr)^\theta.
\]
For $z\in[0,1]$ and $\bar\theta\geq\theta$,
$z^{\bar\theta}\leq z^\theta$. Together with
$\underline c\leq c$, this proves
\eqref{eq:grid-stochastic-nesting}.
\end{proof}

\begin{remark}[Pointwise geometry and stochastic nesting]
\label{rem:grid-common-region}
The original description $(\theta,c)$ supplies the pointwise KL inequality
along the neighboring trajectory in the common region $\mathcal V$.
Lemma~\ref{lem:grid-nesting} records the corresponding one-sided stochastic
nesting property whenever the expected gap is at most one. The restart
transfer below uses the stronger direct interface: the rounded schedule first
produces a stationarity target at the new stage output, and the original
pointwise description then converts that target into the next expected-gap
bound. Thus, no pointwise bound on realized envelope gaps is required.
\end{remark}

Let
\[
 \kappa=(\theta,c,\Delta),
 \qquad
 \bar\kappa=(\bar\theta,\underline c,\bar\Delta).
\]
We call $\bar\kappa$ \emph{recursion-valid relative to $\kappa$} when
\[
 \bar\theta\geq\theta,\qquad
 \underline c\leq c,\qquad
 \bar\Delta\geq\Delta,
\]
the original pointwise KL description $\kappa$ holds on a common region
containing the trajectory generated by $\bar\kappa$, and
$\bar\Delta\leq1$. Since $\Delta\geq\Delta_0$, the inequality
$\bar\Delta\geq\Delta$ also gives the induction base
$\Delta_0\leq\bar\Delta$.

\subsection{Explicit one-sided grids}
\label{app:grid:construction}

Set
\begin{equation}
 \eta_T=\frac{1}{\lceil\log(eT)\rceil}.
 \label{eq:grid-mesh}
\end{equation}
Define
\begin{align}
 \Theta_T
 &=
 \left\{
 \frac12+j\eta_T:
 0\leq j<
 \left\lceil\frac{1}{2\eta_T}\right\rceil
 \right\}
 \cup\{1\},
 \label{eq:theta-grid}\\
 \mathcal C_T
 &=
 \{c_{\min,T}2^j:0\leq j\leq J_c\},
 \qquad
 J_c=
 \left\lfloor
 \log_2\frac{c_{\max,T}}{c_{\min,T}}
 \right\rfloor,
 \label{eq:c-grid}\\
 \mathcal D_T
 &=
 \{\Delta_{\min,T}2^j:0\leq j\leq J_\Delta\}
 \cup\{1\},
 \qquad
 J_\Delta=
 \left\lfloor
 \log_2\frac1{\Delta_{\min,T}}
 \right\rfloor.
 \label{eq:delta-grid}
\end{align}
Duplicate endpoint $1$ in \eqref{eq:delta-grid} is removed. The endpoint
$c_{\max,T}$ need not be inserted: every $c$ in the declared interval has a
dyadic lower neighbor in $\mathcal C_T$.

For a continuous description
$\kappa=(\theta,c,\Delta)$ in \eqref{eq:grid-ranges}, define
\begin{equation}
 \bar\theta
 =
 \min\{v\in\Theta_T:v\geq\theta\},
 \qquad
 \underline c
 =
 \max\{v\in\mathcal C_T:v\leq c\},
 \qquad
 \bar\Delta
 =
 \min\{v\in\mathcal D_T:v\geq\Delta\}.
 \label{eq:grid-neighbor-definition}
\end{equation}
The endpoint conventions guarantee that all three neighbors are well defined.

\begin{lemma}[Neighbor directions and cardinality]
\label{lem:grid-cardinality}
The neighbor in \eqref{eq:grid-neighbor-definition} satisfies
\begin{equation}
 0\leq\bar\theta-\theta\leq\eta_T,
 \qquad
 \frac c2<\underline c\leq c,
 \qquad
 \Delta\leq\bar\Delta<2\Delta.
 \label{eq:grid-neighbor-directions}
\end{equation}
Moreover, for
\[
 \mathcal G_T
 =
 \Theta_T\times\mathcal C_T\times\mathcal D_T,
 \qquad
 K_T=|\mathcal G_T|,
\]
we have
\begin{equation}
 K_T=O((\log T)^3),
 \label{eq:grid-cardinality}
\end{equation}
where the hidden constant depends only on $(r_c,r_\Delta)$.
\end{lemma}

\begin{proof}
Successive nonterminal points of $\Theta_T$ differ by $\eta_T$, and insertion
of the endpoint $1$ can only shorten the final interval. This proves the
exponent inequalities. Consecutive dyadic points in $\mathcal C_T$ and
$\mathcal D_T$ have ratio two, giving the remaining inequalities, including
the capped endpoint of $\mathcal D_T$.

For cardinality,
\begin{align*}
 |\Theta_T|
 &\leq
 \left\lceil\frac{1}{2\eta_T}\right\rceil+1,\\
 |\mathcal C_T|
 &\leq
 1+\log_2(c_{\max,T}/c_{\min,T}),\\
 |\mathcal D_T|
 &\leq
 2+\log_2(1/\Delta_{\min,T}).
\end{align*}
Using \eqref{eq:grid-polynomial-ranges} and
$\eta_T^{-1}=\lceil\log(eT)\rceil$ proves
\eqref{eq:grid-cardinality}, including the degenerate cases
$r_c=0$ and $r_\Delta=0$.
\end{proof}

The one-sided inequalities in
Lemma~\ref{lem:grid-cardinality}, together with the original pointwise KL
description on the common region containing the neighboring trajectory,
make $\bar\kappa=(\bar\theta,\underline c,\bar\Delta)$ recursion-valid
relative to $\kappa$. Lemma~\ref{lem:grid-restart-transfer} below shows that
these conditions are sufficient for the rounded restart schedule to inherit
the capped guarantee.

\subsection{Stage-constant comparison}
\label{app:grid:stageconstant}

Recall from Appendix~\ref{app:restart} that
\[
 p_\theta=4\theta-1,
 \qquad
 A_\kappa
 =
 16e\,c^{-4}\rho L^2\Delta^{-p_\theta}.
\]
For the neighbor
$\bar\kappa=(\bar\theta,\underline c,\bar\Delta)$, write
\[
 u_c=\frac{c}{\underline c}\in[1,2),
 \qquad
 v_\Delta=\frac{\bar\Delta}{\Delta}\in[1,2),
 \qquad
 d=\bar\theta-\theta\in[0,\eta_T].
\]
Direct cancellation gives
\begin{equation}
 \frac{A_{\bar\kappa}}{A_\kappa}
 =
 u_c^4
 v_\Delta^{-(4\bar\theta-1)}
 \Delta^{-4d}.
 \label{eq:A-ratio-exact}
\end{equation}

\begin{lemma}[Uniform stage-constant comparison]
\label{lem:grid-A}
Let
\begin{equation}
 C_A=16e^{4r_\Delta}.
 \label{eq:CA-definition}
\end{equation}
Then
\begin{equation}
 A_{\bar\kappa}\leq C_AA_\kappa,
 \qquad
 1+A_{\bar\kappa}
 \leq
 C_A(1+A_\kappa).
 \label{eq:grid-A-comparison}
\end{equation}
\end{lemma}

\begin{proof}
Because $v_\Delta\geq1$ and $4\bar\theta-1>0$, the
$v_\Delta$ factor in \eqref{eq:A-ratio-exact} is at most one. Also
$u_c^4<16$. Finally,
$\Delta\geq T^{-r_\Delta}$ and
$d\leq1/\lceil\log(eT)\rceil$ imply
% \begin{align*}
%  \Delta^{-4d}
%  &\leq
%  \exp\!\left(
%  \frac{4r_\Delta\log T}
%  {\lceil\log(eT)\rceil}
%  \right)\\
%  &\leq e^{4r_\Delta}.
% \end{align*}
\[
  \Delta^{-4d} \le \exp\!\left( \frac{4r_\Delta\log T}{\lceil\log(eT)\rceil} \right) \le e^{4r_\Delta}.
\]
Hence $A_{\bar\kappa}\leq C_AA_\kappa$. Since $C_A\geq1$,
\[
 1+A_{\bar\kappa}
 \leq
 1+C_AA_\kappa
 \leq
 C_A(1+A_\kappa),
\]
which proves the second inequality.
\end{proof}

The shrinking exponent mesh in \eqref{eq:grid-mesh} keeps the stage-constant
distortion uniform in $T$. In contrast, a fixed positive exponent mesh would
allow $\Delta^{-4d}$ to grow polynomially when the declared gap floor
decreases polynomially with the horizon.

\subsection{Restart guarantee for a recursion-valid neighbor}
\label{app:grid:restart-transfer}

The original KL description closes the rounded restart induction. The
one-sided parameter inequalities ensure that the stationarity target generated
by the rounded schedule is strong enough to recover the rounded expected-gap
target through the original stochastic KL inequality. The following lemma
formalizes this transfer under an arbitrary external cap.

\begin{lemma}[Restart bound under recursion validity]
\label{lem:grid-restart-transfer}
Let $\kappa=(\theta,c,\Delta)$ be a valid continuous description and let
$\bar\kappa=(\bar\theta,\underline c,\bar\Delta)$ be recursion-valid relative
to $\kappa$. Run the restart schedule of Appendix~\ref{app:restart} with
parameters $\bar\kappa$. If its first stage is affordable under an integer cap
$u$, then its completed-stage output $x_u(\bar\kappa)$ satisfies
\begin{equation}
 \bigl(\E S(x_u(\bar\kappa))^2\bigr)^{1/2}
 \leq
 \mathcal R_u(\bar\kappa),
 \label{eq:grid-recursion-restart}
\end{equation}
where
\begin{equation}
 \mathcal R_u(\bar\kappa)
 =
 e^{1/4}\underline c\,\bar\Delta^{\bar\theta}
 \left[
 1+
 \frac{b_{\bar\theta}u}
 {1+A_{\bar\kappa}}
 \right]^{-\nu_{\bar\theta}}.
 \label{eq:grid-recursion-rate}
\end{equation}
\end{lemma}

\begin{proof}
Set
\[
 \bar\gamma=\frac{1}{4\bar\theta},
 \qquad
 \bar D_k=\bar\Delta e^{-\bar\gamma k}.
\]
We prove
\[
 \E Z(x_k)\leq\bar D_k
\]
by induction over completed stages. The base case follows from
\[
 Z(x_0)=\Delta_0\leq\Delta\leq\bar\Delta=\bar D_0.
\]

Suppose that
\[
 \E Z(x_k)\leq\bar D_k.
\]
Applying Lemma~\ref{lem:deterministic-target} with the rounded schedule
parameters gives
\[
 \E S(x_{k+1})^2
 \leq
 \frac{4\sqrt{\rho}L\sqrt{\bar D_k}}
 {\sqrt{\bar t_k+1}},
\]
where
\[
 \bar t_k
 =
 \left\lceil
 A_{\bar\kappa}
 e^{(4\bar\theta-1)\bar\gamma k}
 \right\rceil.
\]
Using
\[
 A_{\bar\kappa}
 =
 16e\,\underline c^{-4}\rho L^2
 \bar\Delta^{-(4\bar\theta-1)}
\]
and the same cancellation as in
Lemma~\ref{lem:restart-induction} yields
\begin{equation}
 \E S(x_{k+1})^2
 \leq
 \underline c^{\,2}
 \bar\Delta^{2\bar\theta}
 e^{-2\bar\theta\bar\gamma(k+1)}
 =
 \underline c^{\,2}
 \bar D_{k+1}^{\,2\bar\theta}.
 \label{eq:grid-rounded-stage-stationarity}
\end{equation}

The original description $\kappa$ holds pointwise on the common region
containing the neighboring trajectory. Hence
Lemma~\ref{lem:stochastic-kl}, applied with the original parameters
$(\theta,c)$ at $x_{k+1}$, gives
\[
 \E Z(x_{k+1})
 \leq
 c^{-1/\theta}
 \bigl(\E S(x_{k+1})^2\bigr)^{1/(2\theta)}.
\]
Combining this inequality with
\eqref{eq:grid-rounded-stage-stationarity},
\[
 \E Z(x_{k+1})
 \leq
 \left(\frac{\underline c}{c}\right)^{1/\theta}
 \bar D_{k+1}^{\bar\theta/\theta}.
\]
Since
\[
 \underline c\leq c,
 \qquad
 \bar\theta\geq\theta,
 \qquad
 0<\bar D_{k+1}\leq\bar\Delta\leq1,
\]
we obtain
\[
 \E Z(x_{k+1})
 \leq
 \bar D_{k+1}.
\]
This closes the rounded expected-gap induction.

Equation~\eqref{eq:grid-rounded-stage-stationarity} also gives, for every
completed stage $k\geq1$,
\[
 \bigl(\E S(x_k)^2\bigr)^{1/2}
 \leq
 \underline c\,\bar\Delta^{\bar\theta}
 e^{-\bar\theta\bar\gamma k}.
\]
The completed-cost and external-cap arguments of
Appendix~\ref{app:restart} are algebraic in the stage schedule and therefore
apply with the rounded parameters. They yield
\eqref{eq:grid-recursion-restart}--\eqref{eq:grid-recursion-rate}.
\end{proof}

\begin{remark}[Description validity and recursion validity]
\label{rem:grid-no-pointwise-transfer}
The master theorem compares against valid continuous KL descriptions, while
the grid serves as an executable approximation device. The original
description supplies the pointwise geometric inequality on the common region;
the grid neighbor supplies the executable restart schedule. The proof of
Lemma~\ref{lem:grid-restart-transfer} shows directly that the rounded
stationarity target is sufficient to close the expected-gap recursion through
the original KL description. Thus pointwise KL validity of the rounded triple
itself is not required.
\end{remark}

\subsection{Constant-factor rate comparison}
\label{app:grid:comparison}

For a continuous description $\kappa$, define
\begin{equation}
 b_\theta=e^{1-1/(4\theta)}-1,
 \qquad
 \nu_\theta=\frac{\theta}{4\theta-1},
 \qquad
 Q_u(\kappa)
 =
 1+\frac{b_\theta u}{1+A_\kappa}.
 \label{eq:grid-b-nu-Q}
\end{equation}
The known-description cap bound is
\begin{equation}
 \mathcal R_u(\kappa)
 =
 e^{1/4}c\Delta^\theta
 Q_u(\kappa)^{-\nu_\theta}.
 \label{eq:grid-rate-definition}
\end{equation}

On $[1/2,1]$, $b_\theta$ is increasing and $\nu_\theta$ is decreasing.
Indeed,
\begin{equation}
 \nu_\theta'
 =
 -\frac{1}{(4\theta-1)^2},
 \qquad
 |\nu_\theta'|\leq1.
 \label{eq:nu-derivative}
\end{equation}
Thus
\begin{equation}
 0\leq
 \nu_\theta-\nu_{\bar\theta}
 \leq
 \bar\theta-\theta
 \leq
 \eta_T.
 \label{eq:nu-mismatch}
\end{equation}

\begin{lemma}[Constant-factor joint approximation]
\label{lem:grid-comparison}
For every integer cap $1\leq u\leq T$,
\begin{equation}
 \mathcal R_u(\bar\kappa)
 \leq
 C_{\rm grid}\mathcal R_u(\kappa),
 \qquad
 C_{\rm grid}
 =
 2e^2C_A^{1/2}
 =
 8e^{2+2r_\Delta}.
 \label{eq:grid-rate-comparison}
\end{equation}
\end{lemma}

\begin{proof}
Since $b_{\bar\theta}\geq b_\theta$,
Lemma~\ref{lem:grid-A} gives
% \begin{align}
%  Q_u(\bar\kappa)
%  &=
%  1+
%  \frac{b_{\bar\theta}u}
%  {1+A_{\bar\kappa}}
%  \nonumber\\
%  &\geq
%  1+
%  \frac{b_\theta u}
%  {C_A(1+A_\kappa)}
%  \nonumber\\
%  &\geq
%  \frac{Q_u(\kappa)}{C_A}.
%  \label{eq:grid-Q-comparison}
% \end{align}
\begin{equation}
  Q_u(\bar\kappa) = 1 + \frac{b_{\bar\theta}u}{1+A_{\bar\kappa}} \ge 1 + \frac{b_\theta u}{C_A(1+A_\kappa)} \ge \frac{Q_u(\kappa)}{C_A}.
  \label{eq:grid-Q-comparison}
\end{equation}
Moreover, using \eqref{eq:grid-neighbor-directions} and $\Delta\leq1$,
% \begin{align}
%  \frac{\underline c\,\bar\Delta^{\bar\theta}}
%       {c\Delta^\theta}
%  &\leq
%  v_\Delta^{\bar\theta}
%  \Delta^{\bar\theta-\theta}
%  \nonumber\\
%  &\leq2.
%  \label{eq:grid-prefactor-comparison}
% \end{align}
\begin{equation}
  \frac{\underline{c}\,\bar\Delta^{\bar\theta}}{c\Delta^\theta}
  \le v_\Delta^{\bar\theta}\Delta^{\bar\theta-\theta}
  \le 2.
  \label{eq:grid-prefactor-comparison}
\end{equation}
Combining
\eqref{eq:grid-Q-comparison}--\eqref{eq:grid-prefactor-comparison},
\begin{equation}
 \frac{\mathcal R_u(\bar\kappa)}
      {\mathcal R_u(\kappa)}
 \leq
 2C_A^{\nu_{\bar\theta}}
 Q_u(\kappa)^{\nu_\theta-\nu_{\bar\theta}}.
 \label{eq:grid-rate-ratio}
\end{equation}
Because $\nu_{\bar\theta}\leq1/2$,
\[
 C_A^{\nu_{\bar\theta}}
 \leq
 C_A^{1/2}.
\]
Also $b_\theta\leq e^{3/4}-1$ and $1+A_\kappa\geq1$, so for $u\leq T$,
\[
 Q_u(\kappa)
 \leq
 1+(e^{3/4}-1)T
 \leq
 e^{3/4}T.
\]
By \eqref{eq:grid-mesh} and \eqref{eq:nu-mismatch},
% \begin{align*}
%  Q_u(\kappa)^{\nu_\theta-\nu_{\bar\theta}}
%  &\leq
%  \exp\!\left(
%  \frac{\log(e^{3/4}T)}
%  {\lceil\log(eT)\rceil}
%  \right)\\
%  &\leq e.
% \end{align*}
\[
  Q_u(\kappa)^{\nu_\theta-\nu_{\bar\theta}}
  \le \exp\!\left( \frac{\log(e^{3/4}T)}{\lceil\log(eT)\rceil} \right)
  \le e.
\]
Hence
\[
 \mathcal R_u(\bar\kappa)
 \leq
 2eC_A^{1/2}\mathcal R_u(\kappa)
 \leq
 2e^2C_A^{1/2}\mathcal R_u(\kappa).
\]
Finally,
\[
 C_A^{1/2}=4e^{2r_\Delta},
\]
and therefore
\[
 C_{\rm grid}
 =
 2e^2C_A^{1/2}
 =
 8e^{2+2r_\Delta}.
\]
\end{proof}

Lemmas~\ref{lem:grid-cardinality},
\ref{lem:grid-A},
\ref{lem:grid-restart-transfer}, and
\ref{lem:grid-comparison} establish the recursion-validity and
constant-factor approximation components of Lemma~\ref{lem:grid}.
Lemma~\ref{lem:grid-nesting} records the associated expectation-level
one-sided KL relation.

\subsection{Affordability transfer}
\label{app:grid:affordability}

For a per-candidate integer training cap $n$, define
\begin{equation}
 \mathfrak K_T^{\rm aff}(f;n)
 =
 \left\{
 \kappa:
 \begin{array}{l}
 \kappa\text{ satisfies Assumption~\ref{ass:klc} and the ranges} 
 \text{ in \eqref{eq:grid-ranges}, and }
 C_A(1+A_\kappa)\leq n
 \end{array}
 \right\}.
 \label{eq:grid-affordable-class}
\end{equation}

\begin{lemma}[First-stage affordability transfers to the neighbor]
\label{lem:grid-affordability}
If
$\kappa\in\mathfrak K_T^{\rm aff}(f;n)$,
then its one-sided neighbor obeys
\begin{equation}
 \left\lceil A_{\bar\kappa}\right\rceil
 \leq n.
 \label{eq:grid-neighbor-affordable}
\end{equation}
Consequently, Lemma~\ref{lem:grid-restart-transfer} applies to the grid
candidate trained under cap $n$.
\end{lemma}

\begin{proof}
Lemma~\ref{lem:grid-A} and $C_A\geq1$ give
\[
 \lceil A_{\bar\kappa}\rceil
 \leq
 1+A_{\bar\kappa}
 \leq
 C_A(1+A_\kappa)
 \leq n.
\]
The one-sided neighbor conditions and the common-region validity of the
original description make $\bar\kappa$ recursion-valid relative to $\kappa$.
Lemma~\ref{lem:grid-restart-transfer} therefore applies.
\end{proof}

Combining Lemmas~\ref{lem:grid-affordability},
\ref{lem:grid-restart-transfer}, and \ref{lem:grid-comparison} gives, for
every $\kappa\in\mathfrak K_T^{\rm aff}(f;n)$,
\begin{equation}
 \bigl(\E S(x_n(\bar\kappa))^2\bigr)^{1/2}
 \leq
 \mathcal R_n(\bar\kappa)
 \leq
 C_{\rm grid}\mathcal R_n(\kappa).
 \label{eq:grid-affordable-export}
\end{equation}

\subsection{Boundary cases and scope}
\label{app:grid:audit}

\paragraph{Grid endpoints.}
At $\theta=1$, $c=c_{\min,T}$, or $\Delta=1$, the corresponding neighbor may
coincide with the continuous description, and all comparison inequalities
remain valid. The dyadic bounds
\[
 \frac c2<\underline c,
 \qquad
 \bar\Delta<2\Delta
\]
also remain valid under the endpoint conventions.

\paragraph{Declared parameter ranges.}
The comparison class consists of descriptions satisfying
\[
 \Delta\in[\Delta_{\min,T},1],
 \qquad
 c\in[c_{\min,T},c_{\max,T}].
\]
The polynomially expanding ranges make the grid logarithmic in the horizon
while allowing increasingly broad geometric descriptions as $T$ grows.

\paragraph{Normalization.}
The upward exponent direction uses deterministic expected-gap targets in
$[0,1]$, which is ensured by the normalized upper-gap coordinate
$\Delta\leq1$. The construction places no corresponding upper bound on each
realized envelope gap. In the restart transfer, the rounded stationarity
target is converted into the next expected-gap target using the original
pointwise KL description.

\paragraph{Nonunique descriptions.}
The construction applies independently to every valid continuous
description. The master theorem can therefore optimize over the entire
affordable valid class rather than selecting or estimating a distinguished
``true'' KL exponent.

\paragraph{Executable grid.}
All grid schedules are executable from their declared parameters. The oracle
comparison requires only that each continuous comparator entering the
infimum admit the recursion-valid neighbor constructed above. Thus validity
is a comparator property rather than an input required to run the portfolio.

\paragraph{Dependence on $r_\Delta$.}
The comparison constant
\[
 C_{\rm grid}=8e^{2+2r_\Delta}
\]
is uniform in the horizon. Its dependence on $r_\Delta$ quantifies the cost
of allowing the declared gap floor to decrease as $T^{-r_\Delta}$.

\paragraph{Small horizons.}
The grid is defined for every $T\geq1$. The oracle comparison becomes
effective as soon as the neighboring first stage of a valid description fits
the allocated training cap, exactly as encoded by
$\mathfrak K_T^{\rm aff}(f;n)$.

\subsection{Interface to the master theorem}
\label{app:grid:contract}

For every
$\kappa\in\mathfrak K_T^{\rm aff}(f;n)$,
this appendix constructs a grid neighbor
$\bar\kappa\in\mathcal G_T$ satisfying
\begin{equation}
 \bigl(\E S(x_n(\bar\kappa))^2\bigr)^{1/2}
 \leq
 \mathcal R_n(\bar\kappa)
 \leq
 C_{\rm grid}\mathcal R_n(\kappa),
 \label{eq:grid-export}
\end{equation}
and hence, by Cauchy--Schwarz,
\begin{equation}
 \E S(x_n(\bar\kappa))
 \leq
 C_{\rm grid}\mathcal R_n(\kappa).
 \label{eq:grid-export-mean}
\end{equation}
The structural quantities are
\[
 K_T=O((\log T)^3),
 \qquad
 C_A=16e^{4r_\Delta},
 \qquad
 C_{\rm grid}=8e^{2+2r_\Delta}.
\]
Thus, the master theorem receives a logarithmic-size executable portfolio and,
for every affordable valid continuous comparator, a grid candidate whose
finite-horizon stationarity guarantee differs by only a uniform multiplicative
constant.

\subsection{Reduced grid with a common conservative gap bound}
\label{app:grid:reduced}

A common conservative upper bound on the initial envelope gap allows the gap
coordinate to be shared across all schedules. Suppose
\[
 0<\Delta_+\leq1,
 \qquad
 \Delta_0\leq\Delta_+,
\]
and define
\[
 \mathcal G_T^{\rm red}
 =
 \Theta_T\times\mathcal C_T,
 \qquad
 K_T^{\rm red}
 =
 |\mathcal G_T^{\rm red}|.
\]
For a valid comparator
\[
 \kappa=(\theta,c,\Delta),
 \qquad
 \Delta\leq\Delta_+,
\]
choose
\[
 \widetilde\kappa
 =
 (\bar\theta,\underline c,\Delta_+),
\]
where $\bar\theta$ and $\underline c$ are the one-sided neighbors defined
above. Write
\[
 d=\bar\theta-\theta,
 \qquad
 u_c=\frac{c}{\underline c},
 \qquad
 \Gamma_\Delta=\frac{\Delta_+}{\Delta}.
\]
Direct cancellation gives
\begin{equation}
 \frac{A_{\widetilde\kappa}}{A_\kappa}
 =
 u_c^4
 \Gamma_\Delta^{-(4\bar\theta-1)}
 \Delta^{-4d}.
 \label{eq:reduced-A-ratio}
\end{equation}
Using
$u_c<2$, $\Gamma_\Delta\geq1$,
$\Delta\geq T^{-r_\Delta}$, and
$d\leq\eta_T$ yields
\begin{equation}
 A_{\widetilde\kappa}
 \leq
 C_A^{\rm red}A_\kappa,
 \qquad
 C_A^{\rm red}
 =
 16e^{4r_\Delta}.
 \label{eq:reduced-A}
\end{equation}
Therefore, the shared conservative gap bound preserves the same uniform
first-stage affordability comparison as the full construction.

Define
\begin{equation}
 \mathfrak K_{T,\rm red}^{\rm aff}(f;n)
 =
 \left\{
 \kappa:
 \begin{array}{l}
 \kappa\text{ satisfies Assumption~\ref{ass:klc} and the declared ranges},\\[-1mm]
 \Delta\leq\Delta_+,\quad
 C_A^{\rm red}(1+A_\kappa)\leq n
 \end{array}
 \right\}.
 \label{eq:reduced-affordable-class}
\end{equation}
For every comparator in this class,
\[
 \lceil A_{\widetilde\kappa}\rceil
 \leq
 1+A_{\widetilde\kappa}
 \leq
 C_A^{\rm red}(1+A_\kappa)
 \leq n.
\]
The one-sided parameter conditions and the common-region validity of the
original description make $\widetilde\kappa$ recursion-valid relative to
$\kappa$. Lemma~\ref{lem:grid-restart-transfer} therefore supplies its capped
restart guarantee.

For the rate comparison, monotonicity of $b_\theta$, the exponent-mesh bound
from Lemma~\ref{lem:grid-comparison}, and
% \begin{align}
%  \frac{\underline c\Delta_+^{\bar\theta}}
%       {c\Delta^\theta}
%  &\leq
%  \Gamma_\Delta^{\bar\theta}
%  \Delta^{\bar\theta-\theta}
%  \nonumber\\
%  &\leq
%  \Gamma_\Delta^{\bar\theta}
%  \label{eq:reduced-prefactor}
% \end{align}
\begin{equation}
  \frac{\underline{c}\Delta_+^{\bar\theta}}{c\Delta^\theta}
  \le \Gamma_\Delta^{\bar\theta}\Delta^{\bar\theta-\theta}
  \le \Gamma_\Delta^{\bar\theta}.
  \label{eq:reduced-prefactor}
\end{equation}
give, for every $1\leq u\leq T$,
\begin{equation}
 \mathcal R_u(\widetilde\kappa)
 \leq
 C_{\rm red}
 \Gamma_\Delta^{\bar\theta}
 \mathcal R_u(\kappa),
 \qquad
 C_{\rm red}
 =
 e^2(C_A^{\rm red})^{1/2}
 =
 4e^{2+2r_\Delta}.
 \label{eq:reduced-rate}
\end{equation}
The factor two appearing in $C_{\rm grid}$ is absent because the reduced
construction does not perform dyadic upper rounding of the gap coordinate.

Consequently, every
$\kappa\in\mathfrak K_{T,\rm red}^{\rm aff}(f;n)$
has a reduced-grid candidate satisfying
\begin{equation}
 \bigl(\E S(x_n(\widetilde\kappa))^2\bigr)^{1/2}
 \leq
 C_{\rm red}
 \Gamma_\Delta^{\bar\theta}
 \mathcal R_n(\kappa).
 \label{eq:reduced-export}
\end{equation}
Moreover,
\begin{equation}
 K_T^{\rm red}
 =
 O((\log T)^2).
 \label{eq:reduced-cardinality}
\end{equation}
The same portfolio construction therefore allocates a larger per-candidate
training and certification budget than the full three-dimensional grid. Under
the nondegenerate-budget condition used by the master theorem,
\[
 n\geq\frac{T}{4K_T^{\rm red}},
\]
which yields Corollary~\ref{cor:reduced}.

The reduced construction exposes a direct tradeoff between portfolio size and
gap-scale mismatch. Sharing $\Delta_+$ removes one logarithmic grid dimension,
while the oracle comparison acquires the explicit factor
\[
 \Gamma_\Delta^{\bar\theta}
 =
 \left(\frac{\Delta_+}{\Delta}\right)^{\bar\theta}.
\]
When $\Gamma_\Delta=O(1)$, the polynomial rate is unchanged and the
portfolio-sharing overhead improves from
$O((\log T)^{3\nu_\theta})$ to
$O((\log T)^{2\nu_\theta})$. As a result, a common conservative gap bound can be
shared across schedules at a precisely quantified cost. The explicit mismatch
factor also characterizes the regime in which retaining the full gap grid
remains preferable.
% END appendices/app_c_grid

% BEGIN appendices/app_d_certificate
\section{Primitive Same-Oracle Certificate}
\label{app:certificate}

This appendix proves Proposition~\ref{prop:certificate} and establishes the
certification interface used by CARP-KL. For any trained candidate $x$, the
construction approximates its proximal point using only projection onto $X$
and fresh calls to the primitive stochastic-subgradient oracle. The resulting
scalar statistic estimates
\[
  S(x)=\|\nabla F(x)\|
\]
with an explicit finite-sample confidence radius. Repeating independent
blocks and taking their scalar median yields simultaneous certification of the
entire trained portfolio under the same primitive-oracle model used for
optimization.

Fix a candidate $x\in X$ and define
\begin{equation}
 \Psi_x(y)
 =
 g(y)+\rho\|y-x\|^2+\iota_X(y),
 \qquad
 y_x^\star
 =
 \arg\min_{y\in X}\Psi_x(y).
 \label{eq:prox-subproblem}
\end{equation}
Since $g$ is $\rho$-weakly convex, $\Psi_x$ is $\rho$-strongly convex and
$y_x^\star$ is the unique proximal point at parameter $1/(2\rho)$. Therefore,
\begin{equation}
 S(x)
 =
 \|\nabla F(x)\|
 =
 2\rho\|x-y_x^\star\|.
 \label{eq:moreau-prox-identity}
\end{equation}

For an integer block length $q\geq1$, initialize $y_0=x$ and run
\begin{equation}
 y_{t+1}
 =
 \Pi_X\!\left[
 y_t-\alpha_t
 \bigl(G(y_t,\xi_t)+2\rho(y_t-x)\bigr)
 \right],
 \qquad
 \alpha_t=\frac{2}{\rho(t+32)},
 \label{eq:certificate-update}
\end{equation}
for $t=0,\ldots,q-1$. The block returns the last iterate and the scalar
\begin{equation}
 Z_q(x)=2\rho\|x-y_q\|.
 \label{eq:block-statistic}
\end{equation}
Thus, certification reduces to approximating the strongly convex proximal
subproblem and reading off the distance from its last iterate to the
candidate.

\subsection{Exact proximal localization}
\label{app:cert:localization}

The proximal geometry itself supplies the localization needed by the
certificate, even when $X$ is unbounded.

\begin{lemma}[Proximal self-localization]
\label{lem:prox-localization}
For every $x\in X$,
\begin{equation}
 \|y_x^\star-x\|
 \leq
 \frac{2L}{\rho},
 \qquad
 S(x)\leq4L.
 \label{eq:prox-localization}
\end{equation}
\end{lemma}

\begin{proof}
Optimality of $y_x^\star$ relative to the feasible comparison point $x$ gives
\begin{equation}
 g(y_x^\star)+\rho\|y_x^\star-x\|^2
 \leq
 g(x).
 \label{eq:prox-compare-x}
\end{equation}
Let
\[
 s=y_x^\star-x,
 \qquad
 \bar G(x)=\E_\xi G(x,\xi)\in\partial g(x).
\]
Weak convexity yields
\begin{equation}
 g(y_x^\star)
 \geq
 g(x)+\langle\bar G(x),s\rangle
 -\frac{\rho}{2}\|s\|^2.
 \label{eq:weak-convex-at-x}
\end{equation}
Combining
\eqref{eq:prox-compare-x}--\eqref{eq:weak-convex-at-x},
\[
 \frac{\rho}{2}\|s\|^2
 \leq
 -\langle\bar G(x),s\rangle
 \leq
 \|\bar G(x)\|\|s\|
 \leq
 L\|s\|,
\]
where the final inequality follows from Jensen's inequality and
\eqref{eq:oracle}. Division by $\|s\|$ when $s\neq0$ gives
\[
 \|y_x^\star-x\|
 \leq
 \frac{2L}{\rho};
\]
the case $s=0$ is immediate. Equation~\eqref{eq:moreau-prox-identity} then
gives
\[
 S(x)\leq4L.
\]
\end{proof}

The same localization bound serves two roles below: it controls the auxiliary
oracle moment during certification and gives a deterministic stationarity
bound on the complement of the simultaneous certification event.

\subsection{Strong-convexity drift}
\label{app:cert:drift}

Write
\[
 \bar G(y)=\E_\xi G(y,\xi),
 \qquad
 h_x(y)=\bar G(y)+2\rho(y-x).
\]
The proximal objective generates a uniform drift toward its exact minimizer.

\begin{lemma}[Exact drift]
\label{lem:certificate-drift}
For every $y\in X$,
\begin{equation}
 \langle h_x(y),y-y_x^\star\rangle
 \geq
 \rho\|y-y_x^\star\|^2.
 \label{eq:certificate-drift}
\end{equation}
\end{lemma}

\begin{proof}
Weak convexity at $y$ gives
\[
 g(y_x^\star)
 \geq
 g(y)
 +
 \langle\bar G(y),y_x^\star-y\rangle
 -
 \frac{\rho}{2}\|y-y_x^\star\|^2.
\]
The quadratic identity
% \begin{align*}
%  \rho\|y_x^\star-x\|^2
%  &=
%  \rho\|y-x\|^2
%  +2\rho\langle y-x,y_x^\star-y\rangle\\
%  &\quad
%  +\rho\|y-y_x^\star\|^2
% \end{align*}
\[
  \rho\|y_x^\star-x\|^2 = \rho\|y-x\|^2 + 2\rho\langle y-x, y_x^\star-y\rangle + \rho\|y-y_x^\star\|^2
\]
therefore implies
\begin{equation}
 \Psi_x(y_x^\star)
 \geq
 \Psi_x(y)
 +
 \langle h_x(y),y_x^\star-y\rangle
 +
 \frac{\rho}{2}\|y-y_x^\star\|^2.
 \label{eq:weak-subgradient-psi}
\end{equation}
Rearranging,
\begin{equation}
 \langle h_x(y),y-y_x^\star\rangle
 \geq
 \Psi_x(y)-\Psi_x(y_x^\star)
 +
 \frac{\rho}{2}\|y-y_x^\star\|^2.
 \label{eq:drift-two-parts}
\end{equation}
Since $\Psi_x$ is $\rho$-strongly convex on the closed convex set $X$ and
$y_x^\star$ is its constrained minimizer,
\[
 \Psi_x(y)-\Psi_x(y_x^\star)
 \geq
 \frac{\rho}{2}\|y-y_x^\star\|^2.
\]
Substitution into \eqref{eq:drift-two-parts} proves
\eqref{eq:certificate-drift}.
\end{proof}

\subsection{Self-bounding oracle moment}
\label{app:cert:selfbound}

Let
\[
 d_t=y_t-y_x^\star,
 \qquad
 D_t=\|d_t\|^2.
\]
All conditional expectations below are with respect to the fresh sample
$\xi_t$, conditional on the candidate and the block history through $y_t$.
The oracle moment grows at most affinely with the squared distance to the
proximal point.

\begin{lemma}[State-dependent second moment]
\label{lem:certificate-moment}
For every $t\geq0$,
\begin{equation}
 \E_t
 \|G(y_t,\xi_t)+2\rho(y_t-x)\|^2
 \leq
 66L^2+16\rho^2D_t.
 \label{eq:certificate-moment}
\end{equation}
\end{lemma}

\begin{proof}
The inequality
$\|a+b\|^2\leq2\|a\|^2+2\|b\|^2$
and \eqref{eq:oracle} give
\begin{equation}
 \E_t
 \|G(y_t,\xi_t)+2\rho(y_t-x)\|^2
 \leq
 2L^2+8\rho^2\|y_t-x\|^2.
 \label{eq:moment-first}
\end{equation}
By Lemma~\ref{lem:prox-localization},
\[
 \|y_t-x\|^2
 \leq
 2\|y_t-y_x^\star\|^2
 +
 2\|y_x^\star-x\|^2
 \leq
 2D_t+\frac{8L^2}{\rho^2}.
\]
Substitution into \eqref{eq:moment-first} gives
\[
 2L^2+16\rho^2D_t+64L^2
 =
 66L^2+16\rho^2D_t,
\]
proving the claim.
\end{proof}

This affine self-bound supplies the complete localization mechanism needed on
an unbounded feasible set. The uniform oracle moment assumption controls the
stochastic term, while the state-dependent quadratic term is absorbed by the
strong-convexity drift.

\subsection{Full-constant last-iterate recursion}
\label{app:cert:recursion}

We now combine the drift and moment bounds to obtain an explicit mean-square
approximation rate for the last proximal iterate.

\begin{lemma}[Mean-square proximal approximation]
\label{lem:certificate-ms}
For every integer $q\geq0$, the iterates in
\eqref{eq:certificate-update} satisfy
\begin{equation}
 \E\|y_q-y_x^\star\|^2
 \leq
 \frac{264L^2}{\rho^2(q+32)}.
 \label{eq:certificate-ms}
\end{equation}
\end{lemma}

\begin{proof}
Nonexpansiveness of projection and feasibility of $y_x^\star$ give
% \begin{align*}
%  D_{t+1}
%  &\leq
%  D_t
%  -
%  2\alpha_t
%  \langle
%  G(y_t,\xi_t)+2\rho(y_t-x),
%  d_t
%  \rangle\\
%  &\quad
%  +
%  \alpha_t^2
%  \|G(y_t,\xi_t)+2\rho(y_t-x)\|^2.
% \end{align*}
\[
  D_{t+1} \le D_t - 2\alpha_t \langle G(y_t,\xi_t) + 2\rho(y_t-x), d_t \rangle + \alpha_t^2 \|G(y_t,\xi_t) + 2\rho(y_t-x)\|^2.
\]
Taking conditional expectation and applying
Lemmas~\ref{lem:certificate-drift}--\ref{lem:certificate-moment},
\begin{equation}
 \E_tD_{t+1}
 \leq
 \left(
 1-2\rho\alpha_t+16\rho^2\alpha_t^2
 \right)D_t
 +
 66L^2\alpha_t^2.
 \label{eq:certificate-raw-recursion}
\end{equation}
Set $z=t+32$. Since
\[
 \alpha_t=\frac{2}{\rho z},
\]
we have
\[
 16\rho^2\alpha_t^2
 =
 \frac{64}{z^2}
 \leq
 \frac{2}{z}
 =
 \rho\alpha_t.
\]
The offset $32$ therefore yields
\begin{equation}
 \E_tD_{t+1}
 \leq
 \left(1-\frac{2}{t+32}\right)D_t
 +
 \frac{264L^2}{\rho^2(t+32)^2}.
 \label{eq:certificate-clean-recursion}
\end{equation}

At initialization, Lemma~\ref{lem:prox-localization} gives
\[
 D_0
 =
 \|x-y_x^\star\|^2
 \leq
 \frac{4L^2}{\rho^2}
 \leq
 \frac{264L^2}{32\rho^2}.
\]
Suppose inductively that
\[
 \E D_t
 \leq
 \frac{264L^2}{\rho^2(t+32)}
\]
and again set $z=t+32$. Taking total expectation in
\eqref{eq:certificate-clean-recursion},
% \begin{align*}
%  \E D_{t+1}
%  &\leq
%  \frac{264L^2}{\rho^2}
%  \left[
%  \left(1-\frac{2}{z}\right)\frac1z
%  +
%  \frac1{z^2}
%  \right]\\
%  &=
%  \frac{264L^2}{\rho^2}
%  \frac{z-1}{z^2}\\
%  &\leq
%  \frac{264L^2}{\rho^2(z+1)},
% \end{align*}
\[
  \mathbb{E} D_{t+1} \le \frac{264L^2}{\rho^2} \left[ \left(1-\frac{2}{z}\right)\frac{1}{z} + \frac{1}{z^2} \right] = \frac{264L^2}{\rho^2} \frac{z-1}{z^2} \le \frac{264L^2}{\rho^2(z+1)}.
\]
because $(z-1)(z+1)\leq z^2$. Induction proves
\eqref{eq:certificate-ms}. The bound applies directly to the last iterate
$y_q$.
\end{proof}

\subsection{One-block scalar certificate}
\label{app:cert:block}

The proximal approximation immediately induces a scalar estimator of
stationarity.

\begin{lemma}[Constant-success block]
\label{lem:certificate-block}
For every fixed candidate $x\in X$ and integer $q\geq1$,
\begin{equation}
 \Prob\!\left(
 |Z_q(x)-S(x)|
 >
 \frac{8\sqrt{66}L}{\sqrt{q+32}}
 \right)
 \leq
 \frac14.
 \label{eq:certificate-block}
\end{equation}
\end{lemma}

\begin{proof}
The reverse triangle inequality and
\eqref{eq:moreau-prox-identity} imply pathwise
\[
 |Z_q(x)-S(x)|
 \leq
 2\rho\|y_q-y_x^\star\|.
\]
Lemma~\ref{lem:certificate-ms} therefore gives
\begin{equation}
 \E|Z_q(x)-S(x)|^2
 \leq
 \frac{1056L^2}{q+32}.
 \label{eq:block-second-moment}
\end{equation}
Applying Markov's inequality at squared threshold
\[
 \frac{4224L^2}{q+32}
 =
 4\frac{1056L^2}{q+32}
\]
proves \eqref{eq:certificate-block}.
\end{proof}

\subsection{Scalar-median amplification}
\label{app:cert:median}

For $\varepsilon\in(0,1/2)$, define the odd block count
\begin{equation}
 B(\varepsilon)
 =
 2\left\lceil4\log(1/\varepsilon)\right\rceil+1.
 \label{eq:certificate-block-count}
\end{equation}
Run $B(\varepsilon)$ independent blocks of length $q$ and define
\begin{equation}
 \widehat S(x)
 =
 \operatorname{median}
 \{Z_q^{(1)}(x),\ldots,Z_q^{(B(\varepsilon))}(x)\}.
 \label{eq:certificate-median}
\end{equation}

\begin{lemma}[Primitive-oracle median certificate]
\label{lem:certificate-median}
For every fixed $x\in X$,
\begin{equation}
 \Prob\!\left(
 |\widehat S(x)-S(x)|
 >
 r_q
 \right)
 \leq
 \varepsilon,
 \qquad
 r_q
 =
 \frac{8\sqrt{66}L}{\sqrt{q+32}}.
 \label{eq:certificate-median-bound}
\end{equation}
The certificate uses exactly $B(\varepsilon)q$ fresh primitive calls.
\end{lemma}

\begin{proof}
Call a block bad when its scalar error exceeds $r_q$.
Lemma~\ref{lem:certificate-block} bounds its bad probability by $1/4$.
Independence of the blocks and Hoeffding's inequality give
\[
 \Prob\!\left(
 \#\{\text{bad blocks}\}
 \geq
 B(\varepsilon)/2
 \right)
 \leq
 e^{-B(\varepsilon)/8}
 \leq
 \varepsilon.
\]
Outside this event, a strict majority of the scalar observations belong to
\[
 [S(x)-r_q,S(x)+r_q].
\]
Their scalar median belongs to the same interval, proving
\eqref{eq:certificate-median-bound}. The call count follows directly from
the block construction.
\end{proof}

Lemma~\ref{lem:certificate-median} establishes
Proposition~\ref{prop:certificate}. The amplification requires only the
second-moment oracle assumption: independent constant-success blocks convert
the one-block estimate into logarithmic confidence without a sub-Gaussian
noise model.

\subsection{Simultaneous certification of random portfolio candidates}
\label{app:cert:conditional}

Let $\mathcal T$ be the sigma-field generated by the complete candidate
training phase, and let
\[
 x_1,\ldots,x_{K_T}
\]
be the resulting $\mathcal T$-measurable candidates. Certification uses
oracle samples fresh relative to $\mathcal T$ and mutually independent across
blocks within each candidate. Cross-candidate independence is not required
for the simultaneous guarantee.

\begin{lemma}[Simultaneous conditional validity]
\label{lem:certificate-simultaneous}
Certify every $x_j$ with
\[
 B_K=B(\delta/K_T)
\]
blocks of length $q_T\geq1$. Then, almost surely with respect to
$\mathcal T$,
\begin{equation}
 \Prob\!\left(
 \bigcap_{j=1}^{K_T}
 \{|\widehat S_j-S(x_j)|\leq r_{q_T}\}
 \;\middle|\;
 \mathcal T
 \right)
 \geq
 1-\delta.
 \label{eq:certificate-simultaneous}
\end{equation}
\end{lemma}

\begin{proof}
Conditional on $\mathcal T$, every candidate is fixed and every fresh
certification block satisfies the same primitive-oracle assumption with the
same constants. Lemma~\ref{lem:certificate-median} therefore bounds each
conditional failure probability by $\delta/K_T$. A conditional union bound
gives \eqref{eq:certificate-simultaneous}. Integrating the conditional
inequality yields the same probability bound unconditionally.
\end{proof}

For the allocation in Algorithm~\ref{alg:carp},
$m\geq B_K$ implies
\[
 q_T
 =
 \left\lfloor\frac{m}{B_K}\right\rfloor
 \geq1,
 \qquad
 B_Kq_T\leq m.
\]
Thus, every median block is nonempty and the certificate remains within its
per-candidate primitive-call allocation.

\subsection{Boundary cases and scope}
\label{app:cert:scope}

\paragraph{Unbounded feasible sets.}
The certificate requires no diameter bound on $X$.
Lemma~\ref{lem:prox-localization} localizes the exact proximal point, while
Lemma~\ref{lem:certificate-moment} expresses the auxiliary oracle moment as
an affine function of the current squared proximal error. The step schedule
then absorbs the state-dependent term through strong-convexity drift.

\paragraph{Uniform oracle moment.}
The oracle assumption \eqref{eq:oracle} applies uniformly to the auxiliary
iterates generated by \eqref{eq:certificate-update}. This is the same
primitive stochastic-subgradient model used by the optimization procedure
and supplies the moment control required by
Lemma~\ref{lem:certificate-moment}.

\paragraph{Oracle access.}
Each certificate block queries
\[
 G(y_t,\xi_t)
 \quad\text{and}\quad
 \Pi_X
\]
only. Function values, $f^\star$, exact proximal points, Moreau gradients,
and KL parameters are not queried. The known quantities $(\rho,L)$ determine
the certificate steps and confidence radius.

\paragraph{Fresh certification.}
Freshness is imposed relative to the trained portfolio. Conditional on the
training sigma-field $\mathcal T$, this makes every candidate fixed and
allows the one-candidate certificate to be applied directly before a union
bound over the portfolio. This separation also prevents candidate training
from adapting to the certification samples used for final selection.

\paragraph{Projected composite structure.}
The certificate is formulated for
\[
 f=g+\iota_X,
\]
matching the optimization model of the main theorem. Projection onto $X$ is a
deterministic operation and all stochastic information continues to enter
through the primitive oracle \eqref{eq:oracle}.

\paragraph{Failure-event control.}
Lemma~\ref{lem:prox-localization} gives
\[
 S(x)\leq4L
\]
uniformly over every portfolio candidate. Hence the complement of the
simultaneous certification event contributes at most $4L\delta$ to the expectation oracle inequality.

\subsection{Interface to portfolio selection}
\label{app:cert:contract}

For the trained portfolio
$x_1,\ldots,x_{K_T}$, define the simultaneous certification event
\[
 \mathcal E_{\rm cert}
 =
 \bigcap_{j=1}^{K_T}
 \left\{
 |\widehat S_j-S(x_j)|
 \leq
 r_{q_T}
 \right\}.
\]
The results above establish, conditionally on the complete training phase,
\begin{equation}
 \Prob(\mathcal E_{\rm cert}\mid\mathcal T)
 \geq
 1-\delta,
 \qquad
 r_{q_T}
 =
 \frac{8\sqrt{66}L}{\sqrt{q_T+32}},
 \label{eq:certificate-export-event}
\end{equation}
with certification cost
\begin{equation}
 B_Kq_T\leq m
 \label{eq:certificate-export-cost}
\end{equation}
for each candidate. In addition,
\begin{equation}
 0\leq S(x_j)\leq4L
 \qquad
 \text{for every }j.
 \label{eq:certificate-export-uniform}
\end{equation}

Equations~\eqref{eq:certificate-export-event}--%
\eqref{eq:certificate-export-uniform} form the complete certification
interface used by Appendix~\ref{app:master}. On
$\mathcal E_{\rm cert}$, all candidate stationarity values are simultaneously
localized by observable scalar intervals; on its complement, the deterministic
bound $S(x_j)\leq4L$ controls the expectation contribution. Thus candidate
training, certification, and final selection remain within one
primitive-oracle model and one explicitly accounted horizon.
% END appendices/app_d_certificate

% BEGIN appendices/app_e_master
\section{Certified Portfolio and Master Theorem}
\label{app:master}

This appendix proves Theorem~\ref{thm:master} by composing the three
interfaces established above: the finite-cap restart guarantee of
Appendix~\ref{app:restart}, its one-sided discretization in
Appendix~\ref{app:grid}, and the fresh primitive-oracle certificate of
Appendix~\ref{app:certificate}. The proof separates four operations:
candidate construction, simultaneous certification, deterministic selection,
and exact primitive-call accounting.

All selection probabilities are conditional on the complete training
sigma-field $\mathcal T$ introduced in
Appendix~\ref{app:cert:conditional}. Candidate quality remains an
expectation/RMS guarantee. This separation allows the training procedures to
be arbitrarily dependent while retaining simultaneous fresh certification of
their realized outputs.

Recall the allocation
\begin{equation}
 K=K_T,\qquad
 n=m=\left\lfloor\frac{T}{2K}\right\rfloor,\qquad
 B_K=2\left\lceil4\log(K/\delta)\right\rceil+1,\qquad
 q_T=\left\lfloor\frac{m}{B_K}\right\rfloor.
 \label{eq:master-allocation}
\end{equation}
The finite-budget theorem operates under
\begin{equation}
 T\geq4K,\qquad
 m\geq B_K,\qquad
 \mathfrak K_T^{\rm aff}(f;n)\neq\varnothing.
 \label{eq:master-budget-threshold}
\end{equation}

\subsection{Certified upper-confidence selection}
\label{app:master:selection}

The selection argument is deterministic once simultaneous certificate
intervals are available. It requires no independence among candidates or
among their training procedures.

\begin{lemma}[Upper-confidence portfolio selection]
\label{lem:master-selection}
Let $x_1,\ldots,x_K\in X$, write $S_j=S(x_j)$, and suppose estimates
$\widehat S_j$ have deterministic radii $r_j\geq0$. Choose any measurable
tie-broken minimizer
\[
 \widehat j
 \in
 \arg\min_{1\leq j\leq K}
 (\widehat S_j+r_j).
\]
On the simultaneous event
\begin{equation}
 \mathcal E
 =
 \bigcap_{j=1}^K
 \{|\widehat S_j-S_j|\leq r_j\},
 \label{eq:master-selection-event}
\end{equation}
we have, simultaneously for every $j$,
\begin{equation}
 S(x_{\widehat j})
 \leq
 S(x_j)+2r_j.
 \label{eq:master-pathwise-selection}
\end{equation}
In particular, if all radii equal $r$,
\begin{equation}
 S(x_{\widehat j})
 \leq
 \min_{1\leq j\leq K}S(x_j)+2r
 \qquad\text{on }\mathcal E.
 \label{eq:master-equal-radius-selection}
\end{equation}
\end{lemma}

\begin{proof}
On $\mathcal E$, for every $j$,
\[
 S(x_{\widehat j})
 \leq
 \widehat S_{\widehat j}+r_{\widehat j}
 \leq
 \widehat S_j+r_j
 \leq
 S(x_j)+2r_j.
\]
The middle inequality is exactly the selection rule. Taking the minimum gives
\eqref{eq:master-equal-radius-selection}.
\end{proof}

Fresh certification converts this pathwise statement into an expectation
oracle inequality for random trained candidates.

\begin{lemma}[Expectation selection under fresh certificates]
\label{lem:master-expectation-selection}
Let $x_1,\ldots,x_K$ be random, $\mathcal T$-measurable candidates. Suppose
fresh certification satisfies
\[
 \Prob(\mathcal E\mid\mathcal T)\geq1-\delta
\]
almost surely, all certificate radii equal $r$, and
\[
 S(x)\leq4L
 \qquad\text{for every }x\in X.
\]
Then, for every candidate index $j$,
\begin{equation}
 \E S(x_{\widehat j})
 \leq
 \E S(x_j)+2r+4L\delta.
 \label{eq:master-expectation-selection}
\end{equation}
\end{lemma}

\begin{proof}
Lemma~\ref{lem:master-selection} gives, on $\mathcal E$,
\[
 S(x_{\widehat j})
 \leq
 S(x_j)+2r.
\]
On $\mathcal E^c$, Lemma~\ref{lem:prox-localization} gives
\[
 S(x_{\widehat j})\leq4L
\]
for every selected candidate. Hence,
\begin{align*}
 \E S(x_{\widehat j})
 &\leq
 \E[S(x_j)\mathbf 1_{\mathcal E}]
 +2r\Prob(\mathcal E)
 +4L\Prob(\mathcal E^c)
 \leq
 \E S(x_j)+2r+4L\delta,
\end{align*}
where the final inequality follows by integrating
\[
 \Prob(\mathcal E^c\mid\mathcal T)\leq\delta.
\]
\end{proof}

Appendix~\ref{app:certificate} supplies this event with
\begin{equation}
 r=r_{q_T}
 =
 \frac{8\sqrt{66}L}{\sqrt{q_T+32}},
 \label{eq:master-certificate-radius}
\end{equation}
simultaneously over the complete trained portfolio.

\subsection{Exact primitive-oracle accounting}
\label{app:master:budget}

The allocation in \eqref{eq:master-allocation} divides the horizon equally
between candidate training and fresh certification. Floors and incomplete
restart stages remain inside this accounting.

\begin{lemma}[Total primitive-oracle budget]
\label{lem:master-budget}
Under
\eqref{eq:master-allocation}--\eqref{eq:master-budget-threshold}, every
certificate block contains at least one primitive call and CARP-KL uses at
most $T$ primitive stochastic-subgradient calls:
\begin{equation}
 Kn+KB_Kq_T\leq T.
 \label{eq:master-total-budget}
\end{equation}
Moreover,
\begin{equation}
 n\geq\frac{T}{4K}.
 \label{eq:master-n-lower}
\end{equation}
\end{lemma}

\begin{proof}
Since $m\geq B_K$,
\[
 q_T
 =
 \left\lfloor\frac{m}{B_K}\right\rfloor
 \geq1.
\]
The floor definition also gives
\[
 B_Kq_T\leq m.
\]
Therefore,
\[
 Kn+KB_Kq_T
 \leq
 K(n+m)
 =
 2K\left\lfloor\frac{T}{2K}\right\rfloor
 \leq
 T,
\]
which proves \eqref{eq:master-total-budget}.

For \eqref{eq:master-n-lower}, set
\[
 z=\frac{T}{2K}.
\]
The condition $T\geq4K$ gives $z\geq2$, and hence
\[
 n
 =
 \lfloor z\rfloor
 \geq
 z-1
 \geq
 \frac z2
 =
 \frac{T}{4K}.
\]
The budget charges the complete cap $n$ to every candidate even when its last
incomplete restart stage leaves unused calls. Projection and random selection
of an inner iterate do not query the stochastic oracle.
\end{proof}

\subsection{Oracle inequality at the allocated cap}
\label{app:master:captheorem}

The allocated-cap form is the direct composition of the discretization and
certification interfaces. It isolates the statistical cost of certified
selection from the later cost of sharing the horizon among portfolio
candidates.

\begin{theorem}[Certified oracle inequality at the allocated cap]
\label{thm:master-cap}
Under the assumptions of Theorem~\ref{thm:master}, CARP-KL satisfies
% \begin{align}
%  \E S(\widehat x_T)
%  &\leq
%  C_{\rm grid}
%  \inf_{\kappa\in\mathfrak K_T^{\rm aff}(f;n)}
%  \mathcal R_n(\kappa)
%  \nonumber\\
%  &\quad+
%  \frac{16\sqrt{66}L}{\sqrt{q_T+32}}
%  +4L\delta.
%  \label{eq:master-cap-bound}
% \end{align}
\begin{equation}
  \mathbb{E} S(\widehat{x}_T)
  \le
  C_{\mathrm{grid}} \inf_{\kappa \in \mathfrak{K}_T^{\mathrm{aff}}(f;n)} \mathcal{R}_n(\kappa)
  + \frac{16\sqrt{66}L}{\sqrt{q_T+32}}
  + 4L\delta.
  \label{eq:master-cap-bound}
\end{equation}
\end{theorem}

\begin{proof}
Fix any
\[
 \kappa\in\mathfrak K_T^{\rm aff}(f;n).
\]
Appendix~\ref{app:grid} constructs a one-sided grid neighbor
\[
 \bar\kappa\in\mathcal G_T
\]
whose first stage is affordable under cap $n$ and whose restart trajectory is
recursion-valid relative to $\kappa$. The exported grid guarantee
\eqref{eq:grid-affordable-export} gives
\begin{equation}
 \bigl(\E S(x_n(\bar\kappa))^2\bigr)^{1/2}
 \leq
 \mathcal R_n(\bar\kappa)
 \leq
 C_{\rm grid}\mathcal R_n(\kappa).
 \label{eq:master-neighbor-bound}
\end{equation}
By Cauchy--Schwarz,
\begin{equation}
 \E S(x_n(\bar\kappa))
 \leq
 C_{\rm grid}\mathcal R_n(\kappa).
 \label{eq:master-neighbor-firstmoment}
\end{equation}

The simultaneous fresh certificate of
Appendix~\ref{app:certificate} has radius
\[
 r_{q_T}
 =
 \frac{8\sqrt{66}L}{\sqrt{q_T+32}}.
\]
Applying Lemma~\ref{lem:master-expectation-selection} with the neighboring
candidate therefore yields
\[
 \E S(\widehat x_T)
 \leq
 C_{\rm grid}\mathcal R_n(\kappa)
 +
 \frac{16\sqrt{66}L}{\sqrt{q_T+32}}
 +
 4L\delta.
\]
This holds for every affordable valid comparator. Taking the infimum proves
\eqref{eq:master-cap-bound}.
\end{proof}

Theorem~\ref{thm:master-cap} separates the two finite-budget effects:
$\mathcal R_n$ records the optimization quality attainable by an affordable
KL description under its allocated training cap, while the final two terms
record the cost of observable same-oracle selection.

\subsection{Full-horizon comparison and proof of Theorem~\ref{thm:master}}
\label{app:master:theorem}

The main theorem expresses the allocated-cap guarantee relative to a
known-description run receiving the full horizon. The following elementary
comparison makes the portfolio-sharing cost explicit.

\begin{lemma}[Budget-sharing comparison]
\label{lem:master-budget-sharing}
For every description
$\kappa=(\theta,c,\Delta)$
and every $1\leq n\leq T$,
\begin{equation}
 \mathcal R_n(\kappa)
 \leq
 \left(\frac{T}{n}\right)^{\nu_\theta}
 \mathcal R_T(\kappa).
 \label{eq:master-budget-sharing}
\end{equation}
Under \eqref{eq:master-budget-threshold},
\begin{equation}
 \mathcal R_n(\kappa)
 \leq
 (4K_T)^{\nu_\theta}
 \mathcal R_T(\kappa).
 \label{eq:master-fourK}
\end{equation}
\end{lemma}

\begin{proof}
Write
\[
 Q_u
 =
 1+\frac{b_\theta u}{1+A_\kappa}.
\]
Since $n\leq T$,
% \begin{align*}
%  Q_n
%  &=
%  1+\frac nT(Q_T-1)\\
%  &=
%  \left(1-\frac nT\right)
%  +\frac nTQ_T\\
%  &\geq
%  \frac nTQ_T.
% \end{align*}
\[
  Q_n = 1 + \frac{n}{T}(Q_T - 1) = \left(1 - \frac{n}{T}\right) + \frac{n}{T}Q_T \ge \frac{n}{T}Q_T.
\]
Substitution into the definition of $\mathcal R_n$ gives
\[
 \mathcal R_n(\kappa)
 \leq
 \left(\frac{T}{n}\right)^{\nu_\theta}
 \mathcal R_T(\kappa).
\]
Lemma~\ref{lem:master-budget} gives
\[
 \frac{T}{n}\leq4K_T,
\]
which proves \eqref{eq:master-fourK}.
\end{proof}

\begin{proof}[Proof of Theorem~\ref{thm:master}]
Combining Theorem~\ref{thm:master-cap} with
Lemma~\ref{lem:master-budget-sharing},
% \begin{align*}
%  \E S(\widehat x_T)
%  &\leq
%  C_{\rm grid}
%  \inf_{\kappa\in\mathfrak K_T^{\rm aff}(f;n)}
%  \left\{
%  (4K_T)^{\nu_\theta}
%  \mathcal R_T(\kappa)
%  \right\}\\
%  &\quad+
%  \frac{16\sqrt{66}L}{\sqrt{q_T+32}}
%  +4L\delta.
% \end{align*}
\[
  \mathbb{E}S(\widehat{x}_T) \le C_{\mathrm{grid}} \inf_{\kappa \in \mathfrak{K}_T^{\mathrm{aff}}(f;n)} \left\{ (4K_T)^{\nu_\theta} \mathcal{R}_T(\kappa) \right\} + \frac{16\sqrt{66}L}{\sqrt{q_T+32}} + 4L\delta.
\]
The factor $(4K_T)^{\nu_\theta}$ remains inside the infimum because
$\nu_\theta$ is part of the comparator description. The total
primitive-oracle bound follows from Lemma~\ref{lem:master-budget}.
\end{proof}

This decomposition makes the adaptive overhead explicit. Relative to the
full-horizon known-description guarantee, candidate sharing contributes
$(4K_T)^{\nu_\theta}$, certification contributes
$16\sqrt{66}L/\sqrt{q_T+32}$, and certificate failure contributes
$4L\delta$.

\subsection{Asymptotic rate and eventual affordability}
\label{app:master:rate}

We now derive Corollary~\ref{cor:rate} for any fixed admissible KL
description.

\begin{proof}[Proof of Corollary~\ref{cor:rate}]
Fix an admissible
\[
 \kappa=(\theta,c,\Delta)
\]
with finite $A_\kappa$ that lies in the declared expanding ranges for all
sufficiently large $T$. Since
\[
 K_T=O((\log T)^3),
\]
the per-candidate training allocation satisfies
\[
 n
 =
 \left\lfloor\frac{T}{2K_T}\right\rfloor
 \longrightarrow\infty.
\]
The comparison constant
\[
 C_A=16e^{4r_\Delta}
\]
is independent of $T$, so eventually
\[
 C_A(1+A_\kappa)\leq n.
\]
Hence, every fixed admissible description is eventually affordable.

For fixed $\kappa$,
\begin{equation}
 \mathcal R_T(\kappa)
 =
 O(T^{-\nu_\theta}),
 \qquad
 \nu_\theta
 =
 \frac{\theta}{4\theta-1}.
 \label{eq:master-fixed-rate}
\end{equation}
Since
\[
 K_T=O((\log T)^3),
\]
the optimization contribution in Theorem~\ref{thm:master} is
\begin{equation}
 O\!\left(
 (\log T)^{3\nu_\theta}
 T^{-\nu_\theta}
 \right).
 \label{eq:master-optimization-rate}
\end{equation}
Set
\[
 \delta=T^{-1}.
\]
Then,
\[
 B_K
 =
 O(\log(K_TT))
 =
 O(\log T),
 \qquad
 m
 =
 \Theta(T/K_T).
\]
For all sufficiently large $T$,
$m\geq2B_K$, and therefore
\[
 q_T
 =
 \left\lfloor\frac{m}{B_K}\right\rfloor
 \geq
 \frac{m}{2B_K}
 =
 \Omega\!\left(
 \frac{T}{(\log T)^4}
 \right).
\]
Consequently,
\begin{equation}
 \frac{16\sqrt{66}L}{\sqrt{q_T+32}}
 =
 O\!\left(
 \frac{L(\log T)^2}{\sqrt T}
 \right),
 \qquad
 4L\delta
 =
 O(L/T).
 \label{eq:master-certificate-rate}
\end{equation}
Finally,
\[
 \nu_\theta
 \in
 [1/3,1/2].
\]
The certificate therefore decays at least as fast in polynomial order as the
optimization contribution. Combining
\eqref{eq:master-optimization-rate} and
\eqref{eq:master-certificate-rate} gives
\[
 \E
 \|\nabla f_{1/(2\rho)}(\widehat x_T)\|
 =
 \widetilde O\!\left(
 T^{-\theta/(4\theta-1)}
 \right).
\]
At $\theta=1/2$, both contributions have polynomial exponent $1/2$ and the
certificate carries the larger logarithmic factor; this is absorbed by the
tilde notation.
\end{proof}
Thus, CARP-KL preserves the polynomial stationarity exponent of every fixed
valid description once that description becomes affordable, while charging
both candidate construction and observable selection to the same primitive
oracle horizon.

\subsection{Finite-budget interpretation}
\label{app:master:finite}

The finite-budget theorem applies to the affordable valid comparator class
while the executable portfolio itself contains every declared grid schedule.
This distinction gives the algorithm a simple implementation and the theorem
a precise comparator interpretation.
\paragraph{Affordable comparison.}
Theorem~\ref{thm:master} compares against
\[
 \mathfrak K_T^{\rm aff}(f;n),
\]
the valid descriptions whose one-sided neighbors can complete their first
restart stage under the allocated training cap. Every fixed admissible
description enters this class for all sufficiently large horizons.

\paragraph{Executable candidates.}
Every declared grid schedule can be run without knowing whether its
description is valid. A candidate whose first restart stage does not fit the
cap returns $x_0$ and remains certifiable. Such candidates do not enter the
oracle comparison until affordability holds, but they remain legitimate
members of the executable portfolio.

\paragraph{Validity-agnostic selection.}
CARP-KL never identifies which grid descriptions are valid. The proof needs
only one recursion-valid neighbor for each comparator entering the infimum.
All candidate outputs, including those generated by other grid schedules,
satisfy the universal stationarity bound
\[
 S(x)\leq4L
\]
and receive the same fresh certification procedure. Hence the selection rule
is independent of KL validity labels.

\paragraph{Conditional confidence and expected optimization quality.}
The simultaneous certification event is high probability conditional on the
complete training phase. The restart guarantees are RMS and expectation
bounds over optimization randomness. Their composition therefore yields an
expected-stationarity oracle inequality selected through high-confidence
fresh certificates.

\paragraph{Full-horizon comparator.}
The quantity $\mathcal R_T(\kappa)$ represents the theoretical guarantee of the known-description restart
procedure if that description received the full primitive-call horizon.
CARP-KL instead shares the horizon across the portfolio. The factor $(4K_T)^{\nu_\theta}$ makes this resource-sharing cost explicit.

\paragraph{Primitive-call accounting.}
Incomplete restart-stage remainders and unused certification remainders need
not be redistributed for the theorem. They remain within the global budget
\eqref{eq:master-total-budget}. Projection operations are deterministic and
are not counted as stochastic primitive-oracle calls.

\subsection{Reduced portfolio with a shared gap bound}
\label{app:master:reduced}

The same composition immediately yields the reduced-grid oracle inequality
when a common conservative initial-gap bound $\Delta_+$ is available. Let
\[
 K_{\rm red}
 =
 K_T^{\rm red},
 \qquad
 n_{\rm red}
 =
 m_{\rm red}
 =
 \left\lfloor
 \frac{T}{2K_{\rm red}}
 \right\rfloor,
\]
and define
\begin{equation}
 B_{\rm red}
 =
 2\left\lceil
 4\log(K_{\rm red}/\delta)
 \right\rceil+1,
 \qquad
 q_{\rm red}
 =
 \left\lfloor
 \frac{m_{\rm red}}{B_{\rm red}}
 \right\rfloor.
 \label{eq:reduced-master-allocation}
\end{equation}
Assume
\[
 T\geq4K_{\rm red},
 \qquad
 m_{\rm red}\geq B_{\rm red},
 \qquad
 \mathfrak K_{T,\rm red}^{\rm aff}(f;n_{\rm red})
 \neq\varnothing.
\]

For every
\[
 \kappa
 \in
 \mathfrak K_{T,\rm red}^{\rm aff}(f;n_{\rm red}),
\]
Appendix~\ref{app:grid} constructs
\[
 \widetilde\kappa
 =
 (\bar\theta,\underline c,\Delta_+)
\]
with
\begin{equation}
 \bigl(
 \E S(x_{n_{\rm red}}(\widetilde\kappa))^2
 \bigr)^{1/2}
 \leq
 C_{\rm red}
 \Gamma_\Delta^{\bar\theta}
 \mathcal R_{n_{\rm red}}(\kappa).
 \label{eq:reduced-master-candidate}
\end{equation}
Applying the same fresh certificate and selection argument gives
\begin{align}
 \E S(\widehat x_T^{\rm red})
 &\leq
 C_{\rm red}
 \inf_{\kappa\in
 \mathfrak K_{T,\rm red}^{\rm aff}(f;n_{\rm red})}
 \left\{
 \Gamma_\Delta(\kappa)^{\bar\theta(\kappa)}
 \mathcal R_{n_{\rm red}}(\kappa)
 \right\}
 +
 \frac{16\sqrt{66}L}{\sqrt{q_{\rm red}+32}}
 +4L\delta.
 \label{eq:reduced-master-cap}
\end{align}
The budget-sharing comparison then yields
\begin{align}
 \E S(\widehat x_T^{\rm red})
 &\leq
 C_{\rm red}
 \inf_{\kappa\in
 \mathfrak K_{T,\rm red}^{\rm aff}(f;n_{\rm red})}
 \left\{
 \Gamma_\Delta(\kappa)^{\bar\theta(\kappa)}
 (4K_{\rm red})^{\nu_\theta}
 \mathcal R_T(\kappa)
 \right\}
 +
 \frac{16\sqrt{66}L}{\sqrt{q_{\rm red}+32}}
 +4L\delta.
 \label{eq:reduced-master-full}
\end{align}

Since
\[
 K_{\rm red}
 =
 O((\log T)^2),
\]
the resource-sharing factor becomes
\[
 (4K_{\rm red})^{\nu_\theta}
 =
 O((\log T)^{2\nu_\theta}).
\]
Moreover, with $\delta=T^{-1}$,
\[
 B_{\rm red}=O(\log T),
 \qquad
 q_{\rm red}
 =
 \Omega\!\left(
 \frac{T}{(\log T)^3}
 \right),
\]
so the reduced certificate contribution is
\[
 O\!\left(
 \frac{L(\log T)^{3/2}}{\sqrt T}
 \right).
\]
Therefore, whenever
\[
 \Gamma_\Delta
 =
 \frac{\Delta_+}{\Delta}
 =
 O(1),
\]
the reduced portfolio preserves the same polynomial stationarity exponent and
improves the portfolio-sharing logarithmic factor from
\[
 O((\log T)^{3\nu_\theta})
 \quad\text{to}\quad
 O((\log T)^{2\nu_\theta}),
\]
with the gap-scale mismatch represented explicitly by
$\Gamma_\Delta^{\bar\theta}$.

\subsection{Finite-budget consequences for the empirical diagnostics}
\label{app:master:finite-budget-diagnostics}

The oracle inequalities expose two finite-budget mechanisms observed in
Appendix~\ref{app:experiments}. First, the portfolio may already contain a
strong realized candidate while the certification budget is still too small
to distinguish it reliably from nearby candidates. Second, reducing the
portfolio dimension reallocates more primitive calls to each candidate and
certificate, but introduces an explicit gap-scale mismatch. The following
results make these two effects precise.

\begin{proposition}[Finite-budget selection resolution]
\label{prop:selection-resolution}
Let $x_1,\ldots,x_K$ be the trained CARP-KL candidates and let
$\widehat x_T=x_{\widehat j}$ be selected by the fresh equal-radius
certificates of Appendix~\ref{app:certificate}. Then,
\begin{equation}
 \E S(\widehat x_T)
 \leq
 \E\!\left[\min_{1\leq j\leq K}S(x_j)\right]
 +
 \frac{16\sqrt{66}L}{\sqrt{q_T+32}}
 +4L\delta.
 \label{eq:best-realized-selection}
\end{equation}
Moreover, on the simultaneous certificate event, if $j_\star$ is any
minimizer of $S(x_j)$, then every candidate satisfying
\begin{equation}
 S(x_j)>S(x_{j_\star})+2r_{q_T}
 \label{eq:selection-exclusion}
\end{equation}
is excluded by the selector. In particular, if the best realized candidate
is unique and every other candidate is separated from it by more than
$2r_{q_T}$, then the selector recovers it exactly on that event.

Under the allocation in \eqref{eq:master-allocation}, if $m\geq2B_K$, then
\begin{equation}
 2r_{q_T}
 =
 O\!\left(
 L\sqrt{\frac{K\log(K/\delta)}{T}}
 \right).
 \label{eq:selection-resolution-scaling}
\end{equation}
Hence candidate availability and candidate recovery are distinct
finite-budget phenomena: a near-oracle candidate may already be present while
the available certification budget is insufficient to resolve its advantage.
\end{proposition}

\begin{proof}
Let $\mathcal E$ denote the simultaneous certificate event from
Appendix~\ref{app:certificate}. On $\mathcal E$,
Lemma~\ref{lem:master-selection} gives
\begin{equation}
 S(\widehat x_T)
 \leq
 \min_{1\leq j\leq K}S(x_j)+2r_{q_T}.
 \label{eq:selection-event-bound}
\end{equation}
On $\mathcal E^c$, Lemma~\ref{lem:prox-localization} gives
$S(\widehat x_T)\leq4L$. Therefore, pathwise,
\[
 S(\widehat x_T)
 \leq
 \min_j S(x_j)+2r_{q_T}
 +4L\mathbf 1_{\mathcal E^c}.
\]
Taking expectations and using
$\Prob(\mathcal E^c)\leq\delta$ yields
\[
 \E S(\widehat x_T)
 \leq
 \E\!\left[\min_jS(x_j)\right]
 +2r_{q_T}+4L\delta.
\]
Substituting the radius from
\eqref{eq:master-certificate-radius},
\[
 2r_{q_T}
 =
 \frac{16\sqrt{66}L}{\sqrt{q_T+32}},
\]
proves \eqref{eq:best-realized-selection}.
For the exclusion statement, fix
$j_\star\in\arg\min_j S(x_j)$.
On $\mathcal E$, simultaneous certificate validity gives
\[
 \widehat S_{j_\star}+r_{q_T}
 \leq
 S(x_{j_\star})+2r_{q_T},
\]
whereas for every $j$,
\[
 \widehat S_j+r_{q_T}
 \geq
 S(x_j).
\]
Consequently, if
$S(x_j)>S(x_{j_\star})+2r_{q_T}$, then
\[
 \widehat S_j+r_{q_T}
 >
 \widehat S_{j_\star}+r_{q_T},
\]
so candidate $j$ cannot minimize the selection score. If $j_\star$ is unique
and every other candidate satisfies this separation, $j_\star$ is therefore
selected.

It remains to identify the scale of the resolution. Since
$q_T=\lfloor m/B_K\rfloor$ and $m\geq2B_K$,
\[
 q_T
 \geq
 \frac{m}{2B_K}.
\]
Under \eqref{eq:master-allocation},
$m=\lfloor T/(2K)\rfloor$; in the regime $T\geq4K$ used by the master
theorem,
\[
 m\geq\frac{T}{4K}.
\]
Hence,
\[
 q_T
 \geq
 \frac{T}{8KB_K}.
\]
Since $B_K=O(\log(K/\delta))$,
\[
 2r_{q_T}
 =
 \frac{16\sqrt{66}L}{\sqrt{q_T+32}}
 =
 O\!\left(
 L\sqrt{\frac{K\log(K/\delta)}{T}}
 \right),
\]
which proves \eqref{eq:selection-resolution-scaling}.
\end{proof}

\begin{remark}
\label{rem:selection-resolution}
Proposition~\ref{prop:selection-resolution} is a finite-budget recovery
statement, not a claim that the selector must fail whenever candidate gaps are
smaller than $2r_{q_T}$. The threshold is sufficient for exclusion and exact
recovery on the simultaneous certificate event. Below that scale, the
certificate guarantee alone does not force the correct ranking. This is the
regime probed by the candidate-ranking diagnostics in
Appendix~\ref{app:experiments}.
\end{remark}

\begin{proposition}[Reduced-portfolio tradeoff]
\label{prop:reduced-tradeoff}
Fix a valid description $\kappa$ belonging to both the full and reduced
affordable comparator classes. Compare, at the same horizon $T$ and confidence
level $\delta$, the full- and reduced-portfolio upper bounds associated with
this same comparator. Write
\[
 K=K_T,
 \qquad
 K_{\rm red}=K_T^{\rm red},
\]
and let $\bar\theta$ denote the rounded exponent of the reduced neighbor.
If
\begin{equation}
 q_{\rm red}\geq q_T
 \qquad\text{and}\qquad
 \Gamma_\Delta^{\bar\theta}
 \leq
 \frac{C_{\rm grid}}{C_{\rm red}}
 \left(\frac{K}{K_{\rm red}}\right)^{\nu_\theta},
 \label{eq:reduced-bound-improvement-condition}
\end{equation}
then the reduced-portfolio upper bound associated with $\kappa$ is no larger
than the corresponding full-portfolio upper bound.

More generally, after removing the common failure term $4L\delta$, the two
bounds are
\begin{equation}
 \underbrace{
 C_{\rm red}\Gamma_\Delta^{\bar\theta}
 (4K_{\rm red})^{\nu_\theta}\mathcal R_T(\kappa)
 }_{\text{reduced optimization term}}
 +
 \underbrace{
 \frac{16\sqrt{66}L}{\sqrt{q_{\rm red}+32}}
 }_{\text{reduced certification term}}
 \label{eq:reduced-tradeoff-left}
\end{equation}
and
\begin{equation}
 \underbrace{
 C_{\rm grid}(4K)^{\nu_\theta}\mathcal R_T(\kappa)
 }_{\text{full optimization term}}
 +
 \underbrace{
 \frac{16\sqrt{66}L}{\sqrt{q_T+32}}
 }_{\text{full certification term}}.
 \label{eq:reduced-tradeoff-right}
\end{equation}
Thus, portfolio reduction trades the larger per-candidate optimization and
certification budgets created by a smaller portfolio against the explicit
gap-mismatch factor $\Gamma_\Delta^{\bar\theta}$.
\end{proposition}

\begin{proof}
For the fixed comparator $\kappa$, Theorem~\ref{thm:master} gives the
full-portfolio bound
\begin{equation}
 \E S(\widehat x_T)
 \leq
 C_{\rm grid}(4K)^{\nu_\theta}\mathcal R_T(\kappa)
 +
 \frac{16\sqrt{66}L}{\sqrt{q_T+32}}
 +4L\delta.
 \label{eq:full-fixed-comparator-bound}
\end{equation}
The reduced master inequality \eqref{eq:reduced-master-full} gives
\begin{equation}
 \E S(\widehat x_T^{\rm red})
 \leq
 C_{\rm red}\Gamma_\Delta^{\bar\theta}
 (4K_{\rm red})^{\nu_\theta}\mathcal R_T(\kappa)
 +
 \frac{16\sqrt{66}L}{\sqrt{q_{\rm red}+32}}
 +4L\delta.
 \label{eq:reduced-fixed-comparator-bound}
\end{equation}

The second condition in
\eqref{eq:reduced-bound-improvement-condition} is equivalent to
\[
 C_{\rm red}\Gamma_\Delta^{\bar\theta}
 (4K_{\rm red})^{\nu_\theta}
 \leq
 C_{\rm grid}(4K)^{\nu_\theta},
\]
and therefore makes the reduced optimization term no larger than the full
optimization term. Since $q\mapsto(q+32)^{-1/2}$ is decreasing, the first
condition makes the reduced certification term no larger than its full
counterpart. The failure contribution $4L\delta$ is common to both bounds.
Adding the termwise comparisons proves the first claim, while
\eqref{eq:reduced-tradeoff-left}--\eqref{eq:reduced-tradeoff-right} display
the comparison before either sufficient condition is imposed.
\end{proof}

\begin{remark}[Interpretation of the reduced-portfolio comparison]
\label{rem:reduced-tradeoff}
Proposition~\ref{prop:reduced-tradeoff} compares theorem upper bounds for a
fixed valid comparator; it does not assert pointwise or empirical dominance
of CARP-2D over CARP-3D. Removing one portfolio coordinate decreases the
candidate count and reallocates the finite horizon, while the shared
conservative gap bound introduces
$\Gamma_\Delta^{\bar\theta}$. Which effect dominates can therefore depend on
the horizon, comparator, and problem constants, consistent with the mixed
finite-budget behavior reported in Appendix~\ref{app:experiments}.
\end{remark}
% END appendices/app_e_master

% BEGIN appendices/app_f_discussion
\section{Additional Discussion, Scope, and Extensions}
\label{app:discussion}

This appendix records the structural interpretation of the CARP-KL guarantee
and several directions that lie beyond the theorem proved here. The central
object is a finite-horizon oracle comparison over valid KL descriptions:
the method need not identify a unique geometric parameter, and the metric
being certified is exactly the Moreau stationarity measure appearing in the
main result.

\paragraph{Competition among nonunique descriptions.}
A KL triple is a valid geometric description rather than an identifiable
parameter of the objective. A single function may admit several triples
\[
  \kappa=(\theta,c,\Delta)
\]
whose stage scales, prefactors, and polynomial exponents lead to different
finite-horizon guarantees. These guarantees can cross as the horizon varies.
Theorem~\ref{thm:master} therefore competes with the complete finite-budget
expression of every valid affordable description instead of estimating a
distinguished ``true'' exponent. This description-based formulation is
intrinsic to the oracle inequality and allows the preferred comparator to
change with the available budget.

\paragraph{Oracle- and metric-aware comparison.}
Rates obtained under different information models need not be directly
comparable. Smooth or mean-square-smooth KL analyses can exploit regularity
beyond the primitive stochastic-subgradient oracle~\eqref{eq:oracle} and may
obtain stronger objective-gap guarantees
\citep{fatkhullin2022sharp}. Convex sharpness frameworks may use observable
objective progress or deterministic first-order information
\citep{roulet2017sharpness,renegar2022restart,adcock2025restarts}.
CARP-KL instead measures complexity in primitive stochastic-subgradient calls
and targets
\[
  S(x)=\|\nabla f_{1/(2\rho)}(x)\|.
\]
Its rate should therefore be interpreted relative to this oracle and this
stationarity metric.

\paragraph{Stationarity-aligned selection.}
The certification mechanism is deliberately aligned with the performance
criterion of the theorem. Selecting by
\[
  F(x)-f^\star
\]
would require access to both the Moreau envelope and the optimum value, while
selection by $f(x)$ would introduce objective-value access not present in the
primitive oracle and would not directly order Moreau stationarity. The
proximal-distance construction of Appendix~\ref{app:certificate} instead
produces an observable scalar certificate for $S(x)$ from fresh calls to the
same stochastic-subgradient oracle. Thus, optimization and validation operate
under one information model and target the same quantity.

\paragraph{Horizon-free adaptation.}
CARP-KL is stated for a known primitive-call horizon $T$, which determines the
grid resolution and the division of calls between training and certification.
A doubling construction provides a natural route to an anytime version, but
the epochs must jointly preserve one-sided grid coverage, freshness of the
certification samples, and global primitive-call accounting. Establishing an
anytime oracle inequality with these properties is a distinct extension of
the present finite-horizon construction.

\paragraph{General convex composite terms.}
The current model
\[
  f=g+\iota_X
\]
uses projection onto $X$ as the deterministic nonsmooth operation. This
structure enters both the projected stochastic-subgradient restart lemma and
the primitive proximal-distance certificate. Replacing $\iota_X$ by a
general proper closed convex term $h$ would require an explicit primitive
interface for the resulting composite proximal subproblems together with
corresponding restart and certification lemmas. The portfolio principle
remains applicable once such interfaces are available, but the required
composite analysis is separate from the theorem established here.

\paragraph{Adaptive overhead and lower bounds.}
The full portfolio uses
\[
  K_T=O((\log T)^3)
\]
candidates, while a shared conservative initial-gap bound reduces the
portfolio to
\[
  K_T^{\rm red}=O((\log T)^2).
\]
These constructions establish sufficient logarithmic overheads for preserving
the known-description polynomial stationarity exponent. They do not establish
the minimum possible cost of adaptation. Lower bounds for parameter adaptation
in stochastic convex optimization demonstrate that adaptation can carry a
nonconstant price in other settings \citep{carmon2024price}, but they do not
characterize the present problem of competition among unknown KL
descriptions. Determining the optimal portfolio, certification, and
information-theoretic overhead for this class remains an open problem.

\paragraph{Finite-budget certification as a design bottleneck.}
The oracle inequality separates the optimization contribution from the cost
of observable selection:
\[
  \E S(\widehat x_T)
  \leq
  \text{portfolio optimization term}
  +
  \text{certificate radius}
  +
  \text{failure contribution}.
\]
The experiments show why this decomposition matters at practical horizons.
Strong candidates can already be present in the portfolio while finite-sample
certification and budget sharing still affect which candidate is returned.
This observation does not alter the asymptotic polynomial exponent proved by
Theorem~\ref{thm:master}; rather, it identifies certified selection as an
important finite-budget component of adaptive KL optimization.

\paragraph{Outlook.}
The present results establish an executable form of KL adaptation in which
unknown and potentially nonunique geometric descriptions are converted into
a finite candidate portfolio, candidate stationarity is made observable using
the primitive oracle, and the complete optimization--selection procedure is
charged to one horizon. Natural next steps are to sharpen the adaptive
overhead, develop anytime and broader composite variants, and design more
sample-efficient certificates while retaining the same-oracle comparison
principle.
% END appendices/app_f_discussion

% BEGIN appendices/app_g_experiments
\section{Experimental Protocols and Diagnostics}
\label{app:experiments}

This appendix documents the complete experimental protocol underlying
Section~\ref{sec:experiments} and reports the diagnostics used to interpret
CARP-KL at finite horizons. The experiments serve three distinct purposes:
to verify the implementation and primitive-oracle accounting, to evaluate the
finite-budget behavior of certified adaptation, and to identify mechanisms
that are not visible from the asymptotic oracle inequality alone.

The complete pre-results specification and machine-readable configuration are
included in the reproducibility package. Experimental choices designated as
frozen below were fixed before the corresponding outcomes were examined.
Later diagnostics are identified explicitly and do not replace the earlier
criteria.

\subsection{Controlled problems and evaluation metric}
\label{app:experiments:problems}

All primary instances are one dimensional, with
\[
  X=[-0.4,0.4],
  \qquad
  x_0=0.4,
  \qquad
  \rho=1,
\]
and stochastic oracle
\begin{equation}
  G(x,\xi)=f'(x)+0.2\xi,
  \label{eq:experiment-oracle}
\end{equation}
where $\xi$ is an independent Rademacher variable.

We use
\[
  f_p(x)=\frac{|x|^p}{p},
  \qquad
  p\in\{2,3,4\},
\]
with KL exponents
\[
  \frac12,\qquad
  \frac23,\qquad
  \frac34,
\]
and the endpoint instance
\begin{equation}
 f_{\rm flat}(x)
 =
 \begin{cases}
   \exp(-1/x^2), & x\neq0,\\
   0, & x=0,
 \end{cases}
 \label{eq:flat-instance}
\end{equation}
with exponent $1$. The flat instance is convex on the chosen domain.
For each problem,
\[
  L=\max_{x\in X}|f'(x)|+0.2
\]
is an analytic uniform oracle-moment bound. Before stochastic evaluation, a
deterministic preflight calculation verifies that the declared parameter box
contains a valid KL description for each controlled problem.

\paragraph{Offline stationarity metric.}
For evaluation only, the unique proximal point is computed by bracketed
bisection to absolute tolerance $10^{-12}$, giving
\begin{equation}
  S(x)
  =
  2\rho
  \left|
  x-\prox_{f/(2\rho)}(x)
  \right|.
  \label{eq:offline-stationarity}
\end{equation}
This exact numerical evaluation is performed strictly after a method returns
its output and is never supplied to any optimization or selection procedure.

Unless stated otherwise, slope summaries are ordinary least-squares fits to
log median stationarity over the prespecified last four horizons.
Uncertainty intervals use $10{,}000$ paired seed-bootstrap resamples with
frozen seed $8675309$. The raw experimental records additionally retain
actual primitive calls, unused allocation, completed restart stages, and
projection counts.

\subsection{Frozen diagnostic protocol}
\label{app:experiments:frozen}

The declared CARP-KL ranges are
\[
  \theta\in[1/2,1],
  \qquad
  c\in[1/2,2],
  \qquad
  \Delta\in[1/4,1],
\]
with the grids of Appendix~\ref{app:grid} and
\[
  \delta=T^{-1}.
\]

The initial rate-tracking diagnostic uses
\[
  T=2^{15},\ldots,2^{21}
\]
and twenty training seeds $1729,\ldots,1748$. It compares CARP-KL with a
tuned ERdA receiving cap $T$, a tuned ERdA receiving the shared cap $n$, an
aggressively misspecified restart schedule, and no-restart projected
stochastic subgradient descent.

The joint-misspecification diagnostic uses all $3^3$ combinations of exponent
offsets
\[
  \{-\eta_T,0,+\eta_T\}
\]
and multipliers
\[
  \{1/2,1,2\}
\]
for both $c$ and $\Delta$, at $T=2^{20}$.

The certificate-calibration diagnostic uses $200$ fresh certificate
replicates at
\[
  q\in\{16,32,64,128,256,512,1024\}.
\]
The competitive benchmark and subsequent mechanism, robustness, and
extended-budget diagnostics were introduced only after these initial
protocols and their outcomes had been frozen.

\paragraph{Prespecified interpretation criteria.}
The rate-tracking criterion requires the mean absolute CARP-versus-shared-
oracle slope difference across the four instances to be at most $0.10$, no
individual difference to exceed $0.18$, and every CARP slope to be negative.
The misspecification-direction criterion requires at least $75\%$ of matched
conservative/aggressive cells to favor the conservative description in
median terminal stationarity, with no larger fallback frequency.
Certificate scaling requires a fitted median absolute-error slope in
$[-0.65,-0.35]$ for at least three instances and empirical coverage no more
than $0.05$ below nominal. The ranking-resolution criterion requires
selection regret at $q=1024$ to be no larger than at $q=16$ on every instance,
together with pooled Spearman correlation at most $-0.8$.

These thresholds are used only to interpret the frozen diagnostics. Seeds,
fit windows, endpoints, and thresholds are retained when a criterion is not
met.

\subsection{Initial diagnostic results}
\label{app:experiments:initial-results}

\paragraph{Rate tracking.}
The frozen run completed all $560$ problem--horizon--seed jobs and all
$2{,}800$ method records without method failure. The CARP-KL/shared-oracle
slope pairs were
\[
  (-0.552,-0.472)\quad\text{for P2},
\]
\[
  (-0.287,-0.271)\quad\text{for P3},
\]
\[
  (-0.646,-0.092)\quad\text{for P4},
\]
and
\[
  (2.025,0.011)\quad\text{for FLAT}.
\]
The mean absolute slope difference was $0.666$, exceeding the prespecified
$0.10$ threshold. Thus, the frozen experiment does not empirically confirm
rate tracking over these horizons.

A subsequent aggregation decomposition, reported separately from the frozen
criterion, shows that P4 nearly aligns under mean aggregation, with slope
difference $0.009$. FLAT remains numerically unstable under high-precision
offline evaluation. These observations characterize the finite-horizon
behavior but do not replace the original diagnostic.

\paragraph{Joint misspecification.}
The frozen run produced all $2{,}160$ expected records: $27$ nominal cells,
four instances, and twenty seeds. Under the prespecified extreme-pair
comparison, the conservative endpoint had no larger median terminal
stationarity on FLAT and P4, but not on P2 or P3. The resulting fraction was
$0.50$, below the $0.75$ criterion. Fallback frequency was zero for both
endpoints on all four instances.

The conservative/aggressive median stationarity pairs were approximately
\[
  (1.05,1.25)\times10^{-12}
  \qquad\text{for FLAT},
\]
\[
  (7.91\times10^{-4},8.16\times10^{-5})
  \qquad\text{for P2},
\]
\[
  (3.68\times10^{-4},3.17\times10^{-4})
  \qquad\text{for P3},
\]
and
\[
  (4.26\times10^{-5},3.71\times10^{-4})
  \qquad\text{for P4}.
\]
The complete $3^3$ response surfaces are nonmonotone and include intermediate
cells that outperform one or both extreme endpoints. The experiment therefore
shows that the one-sided theoretical comparison should not be interpreted as
a uniform finite-sample ordering of conservative and aggressive schedules.

\paragraph{Certificate calibration.}
All $5{,}600$ expected certificate records were obtained. The median
certificate-error slopes were
\[
  -0.5368\ \text{(P2)},\qquad
  -0.3882\ \text{(P3)},\qquad
  -0.4132\ \text{(P4)},\qquad
  -0.3853\ \text{(FLAT)}.
\]
All four lie in the prespecified certificate-scaling band. Marginal and
simultaneous empirical coverage were $1.000$ for every problem and every
tested block size. These results support the conservative finite-sample
scaling predicted by the certificate analysis.

\paragraph{Ranking resolution.}
The endpoint regret comparison between $q=16$ and $q=1024$ holds on all four
instances. Across the full block-size sweep, however, the pooled Spearman
correlation between block size and selection regret is only $-0.13494$, rather
than the prespecified value at most $-0.8$. Median regret is constant across
the grid for P2 and P4, nearly constant for FLAT, and decreases sharply only
at $q=1024$ for P3.

Thus, the certificate calibration and ranking diagnostics establish two
different empirical facts: the scalar certificate has conservative coverage
and the expected error scaling, while increased certificate resolution does
not produce uniformly monotone ranking improvement over the tested finite
budgets.

\subsection{Competitive benchmark}
\label{app:experiments:competitive}

After the initial protocols and outcomes were frozen, we introduced the
competitive benchmark reported in Section~\ref{sec:experiments}. Comparator
hyperparameters were selected using ten disjoint tuning seeds and fixed before
evaluation.

The benchmark uses twenty evaluation seeds
\[
  3101,\ldots,3120,
\]
seven horizons
\[
  \log_2T\in\{15,17,19,20,21,23,25\},
\]
and six methods:
CARP-KL, Oracle-KL, tuned SGD, logarithmic step-size search, AdaGrad-Norm, and
a source-policy-faithful Gao--Deng specialization. Its protocol-specific
bootstrap uses $10{,}000$ replicates with frozen seed $202710$.

Two independent full executions reproduce all $168$ problem--budget--method
cell medians and all $28$ canonical paired CARP/Oracle median log-ratios.
At $T=2^{20}$, the expanded sweep exactly reproduces the reference snapshot
in Table~\ref{tab:e5-competitive}.

\begin{table}[t]
\centering
\caption{Reference competitive snapshot at $T=2^{20}$ over twenty evaluation
seeds: median offline Moreau stationarity; lower is better.}
\label{tab:e5-competitive}
\small
\setlength{\tabcolsep}{5pt}
\begin{tabular}{lrrrrrr}
\toprule
Problem & CARP-KL & Oracle & SGD & LogSearch & AdaGrad & Gao--Deng \\
\midrule
P2
& $2.98\!\times\!10^{-3}$
& $1.13\!\times\!10^{-4}$
& $2.41\!\times\!10^{-3}$
& $9.42\!\times\!10^{-4}$
& $1.69\!\times\!10^{-3}$
& $3.29\!\times\!10^{-3}$ \\
P3
& $2.21\!\times\!10^{-3}$
& $1.31\!\times\!10^{-4}$
& $1.57\!\times\!10^{-3}$
& $4.88\!\times\!10^{-4}$
& $1.62\!\times\!10^{-4}$
& $8.15\!\times\!10^{-4}$ \\
P4
& $6.55\!\times\!10^{-4}$
& $6.69\!\times\!10^{-5}$
& $3.66\!\times\!10^{-4}$
& $5.14\!\times\!10^{-4}$
& $1.31\!\times\!10^{-4}$
& $1.85\!\times\!10^{-4}$ \\
FLAT
& $5.70\!\times\!10^{-7}$
& $1.14\!\times\!10^{-12}$
& $1.13\!\times\!10^{-12}$
& $1.37\!\times\!10^{-5}$
& $1.18\!\times\!10^{-11}$
& $2.93\!\times\!10^{-11}$ \\
\bottomrule
\end{tabular}
\end{table}

The benchmark is intentionally interpreted as a finite-budget comparison
rather than an empirical theorem verification. CARP-KL is not uniformly best
among the six methods; the experiment quantifies the realized cost of
portfolio sharing and certified selection under the primitive-call accounting
used by the theory.

\subsection{Expanded budget sweep and adaptation gap}
\label{app:experiments:gap}

The complete competitive experiment evaluates all six methods at
\[
  \log_2T\in\{15,17,19,20,21,23,25\}.
\]
CARP-KL displays a pronounced small-budget transition between $2^{15}$ and
$2^{17}$. From $2^{19}$ through $2^{25}$, its median absolute stationarity
decreases monotonically on P2 and P3. P4 improves overall with a local
deterioration at $2^{21}$, while FLAT is strongly influenced by numerical
floors.

For the prespecified seed-paired natural-log diagnostic,
\begin{equation}
 G_p(T)
 =
 \operatorname{median}_s
 \left[
 \log
 \frac{
 S_{\rm CARP,p,s}(T)
 }{
 S_{\rm Oracle,p,s}(T)
 }
 \right],
 \label{eq:paired-gap}
\end{equation}
the canonical trajectory is

\begin{center}
\small
\setlength{\tabcolsep}{5pt}
\begin{tabular}{crrrr}
\toprule
$\log_2T$ & P2 & P3 & P4 & FLAT \\
\midrule
15 & $-0.841$ & $-3.334$ & $-6.452$ & $-6.180$ \\
17 & $ 2.294$ & $ 4.001$ & $ 5.542$ & $22.707$ \\
19 & $ 2.233$ & $ 3.035$ & $ 1.624$ & $ 0.758$ \\
20 & $ 2.770$ & $ 3.138$ & $ 2.073$ & $12.664$ \\
21 & $ 2.160$ & $ 3.799$ & $ 2.850$ & $ 0.988$ \\
23 & $ 2.867$ & $ 3.248$ & $ 2.807$ & $ 8.153$ \\
25 & $ 3.149$ & $ 3.223$ & $ 2.242$ & $ 2.220$ \\
\bottomrule
\end{tabular}
\end{center}

The paired gap remains substantial and does not uniformly contract over the
tested range. FLAT ratios require particular caution because several methods
approach the offline numerical floor. This observation is compatible with the
oracle inequality: Theorem~\ref{thm:master} preserves the polynomial
stationarity exponent up to logarithmic factors and constants; it does not
predict that the finite-horizon CARP/Oracle ratio converges monotonically to
one over the tested budgets.

\begin{figure}[t]
\centering
\includegraphics[width=.78\textwidth]
{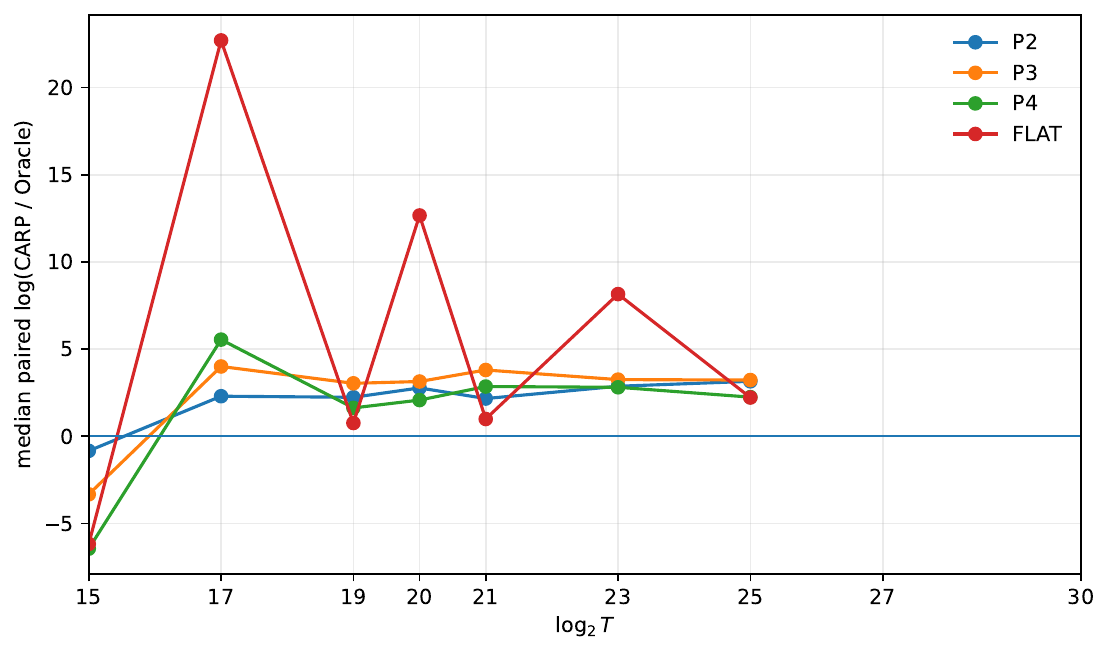}
\caption{\textbf{Paired CARP/Oracle adaptation gap across the competitive
budget sweep.}
Median seed-paired natural-log stationarity ratio over the seven frozen
budgets. Positive values favor Oracle-KL. FLAT is floor-sensitive.}
\label{fig:e5-gap-app}
\end{figure}

\subsection{Reduced-portfolio diagnostic}
\label{app:reduced-experiment}

Corollary~\ref{cor:reduced} predicts a structural tradeoff when a common
conservative gap bound is available: removing the gap-grid coordinate reduces
the number of candidates and increases the optimization and certification
budget available to each remaining schedule, while introducing the explicit
scale-mismatch factor $\Gamma_\Delta^{\bar\theta}.$

To isolate this mechanism, we compare two implementations of CARP-KL.
We call the full three-dimensional portfolio over
\((\theta,c,\Delta)\) \emph{CARP-3D}, and the reduced two-dimensional
portfolio over \((\theta,c)\), with a shared conservative gap bound
\(\Delta_+\), \emph{CARP-2D}. Thus CARP-3D independently discretizes all
three coordinates, whereas CARP-2D removes the gap coordinate from the
portfolio and uses the same value \(\Delta_+\) for every candidate.

We evaluate CARP-3D and CARP-2D on the same four controlled problems,
all twenty frozen evaluation seeds, and
\[
  \log_2 T \in \{15,17,19,20\}.
\]
For CARP-2D, we set
\[
  \Delta_+ = 1,
\]
while retaining the same restart construction, certification procedure,
selection rule, and total primitive-call accounting as CARP-3D. In the
controlled comparison, the valid comparator has
\[
  \Delta=\frac14,
  \qquad
  \Gamma_\Delta
  =\frac{\Delta_+}{\Delta}
  =4.
\]
At $T=2^{20}$, the full grid contains $81$ candidates and the reduced grid
contains $27$. Their resource allocations are
\[
  n_{\rm 3D}=6472,
  \qquad
  q_{\rm 3D}=43,
\]
and
\[
  n_{\rm 2D}=19418,
  \qquad
  q_{\rm 2D}=139.
\]
Thus, the reduced construction directly reallocates the fixed horizon from
portfolio breadth to per-candidate optimization and certification.

The corresponding median offline stationarity values are

\begin{center}
\begin{tabular}{lrr}
\toprule
Problem & CARP-3D & CARP-2D \\
\midrule
P2   & $2.98\times10^{-3}$ & $2.33\times10^{-3}$ \\
P3   & $2.21\times10^{-3}$ & $9.66\times10^{-4}$ \\
P4   & $6.55\times10^{-4}$ & $5.84\times10^{-4}$ \\
FLAT & $5.70\times10^{-7}$ & $1.13\times10^{-6}$ \\
\bottomrule
\end{tabular}
\end{center}

The effect varies with the budget. At $2^{15}$, CARP-2D has larger median
stationarity on all four problems. At $2^{17}$, it substantially improves P3,
P4, and FLAT relative to the CARP-3D small-budget transition. The picture is
mixed at $2^{19}$, while at $2^{20}$ the marginal medians improve on P2--P4
and deteriorate on FLAT. Seed-paired comparisons are likewise mixed.

The diagnostic therefore exhibits the resource-reallocation mechanism
predicted by the reduced construction. Proposition~\ref{prop:reduced-tradeoff}
formalizes the corresponding finite-budget competition between the smaller
portfolio-sharing and certificate terms and the explicit gap-mismatch factor.
The realized outcome also depends on candidate quality and finite-sample
selection.

\subsection{Mechanism diagnostics}
\label{app:mechanism-ablations}

The competitive results motivate a direct examination of the three components
that determine finite-budget adaptation:
\[
  \text{candidate availability},
  \qquad
  \text{certificate resolution},
  \qquad
  \text{candidate selection}.
\]
The diagnostics below evaluate these components using offline information
only after the CARP-KL run has completed.

\paragraph{Strong candidates are present in the realized portfolio.}
Exact deterministic replay evaluates every realized candidate using the
offline stationarity metric \eqref{eq:offline-stationarity}. At the frozen
$T=2^{22}$ mechanism endpoint, the best realized candidate has a negative
median paired log-stationarity ratio relative to Oracle-KL on all four
controlled problems. Thus, the realized CARP-KL portfolio contains candidates
whose post-return stationarity is competitive with the oracle comparator.

The selected candidates need not coincide with these best realized
candidates. Their median true ranks are
\[
  43/90\quad\text{(P2)},\qquad
  59.5/90\quad\text{(P3)},\qquad
  54/90\quad\text{(P4)},\qquad
  38/90\quad\text{(FLAT)}.
\]
The finite-budget gap therefore cannot be attributed solely to an absence of
strong candidates in the realized portfolio.

\paragraph{Certificate validity and ranking resolution are distinct.}
The certificate-calibration experiment shows conservative empirical coverage
and the predicted order of error decay. Exact replay nevertheless shows
substantial selection regret, and increasing the certificate block size does
not yield uniformly monotone ranking improvement. A certificate can therefore
be statistically valid while still being too coarse to reliably distinguish
nearby candidate stationarity levels under a fixed finite budget.
Proposition~\ref{prop:selection-resolution} makes this distinction explicit:
the selected point tracks the best realized candidate only up to the
finite-budget certificate resolution.

\paragraph{Selection interventions exhibit regime dependence.}
Prospectively specified selector interventions can improve candidate recovery
in some regimes and deteriorate it in others. A subsequent
candidate-generation $\times$ restriction $\times$ selector factorial
diagnostic reaches the same conclusion: fresh validation or stabilization
does not yield a uniformly better selector across the controlled problems.
The complete set of positive and negative cells is retained in the
reproducibility package.

\paragraph{Certification competes directly with candidate optimization.}
At $T=2^{22}$, certification consumes approximately half of the
primitive-call horizon, candidate optimization receives approximately
$35$--$37\%$, and integer rounding leaves approximately $13$--$15\%$ unused.
Additional validation can therefore improve ranking resolution while reducing
the calls available to construct each candidate. This is precisely the
finite-budget tradeoff exposed by the additive certificate term and the
portfolio-sharing factor in Theorem~\ref{thm:master}.

\paragraph{Mechanism interpretation.}
Taken together, the replay and ablation results identify certified selection
as a principal \emph{observed} finite-budget bottleneck of the current
construction once strong candidates are available. This statement is
mechanistic rather than causal: the CARP/Oracle gap also reflects portfolio
sharing, finite candidate budgets, rounding, and problem-dependent constants.
The experiments establish that candidate availability alone does not explain
the observed gap and that the tested selector modifications do not provide a
uniform replacement.

\subsection{Finite-second-moment robustness}
\label{app:experiments:heavy-tail}

The theoretical oracle model requires a finite uniform second moment rather
than sub-Gaussian noise. To probe this distinction, we evaluate the frozen
methods under centered variance-one Rademacher noise and standardized
Student-$t$ noise with
$
  5, 3, 2.25
$
degrees of freedom. These distributions remain within the finite-second-
moment regime.

CARP-KL remains operational on P2--P4 across the tested noise families without
systematic catastrophic degradation. FLAT remains sensitive to numerical
floors. Several practical comparators continue to attain lower stationarity
on substantial parts of the experiment, so the diagnostic supports
operational robustness under the stated moment model rather than a
comparative heavy-tail advantage.

% END appendices/app_g_experiments

\end{document}